\documentclass[11pt]{amsart}
\usepackage{amsmath}
\usepackage{mathrsfs}
\usepackage{mathrsfs,amsmath, tikz,verbatim, wasysym,xcolor, amsfonts, amssymb, amsthm, graphicx, hyperref, enumerate, xfrac, xspace, mathtools, amsrefs, comment}
\usetikzlibrary{arrows}
\usetikzlibrary{shapes.geometric}
\usetikzlibrary{matrix}
\usetikzlibrary{decorations.pathreplacing, decorations.markings}
\usepackage[utf8]{inputenc}
\usepackage[top=0.9in, bottom=0.9in, right=0.9in, left=0.9in]{geometry}
\usepackage{hyperref}
\usepackage{graphicx}
\usepackage{xcolor}
\usepackage{tikz-cd}
\usepackage{multicol}

\usepackage[capitalise]{cleveref}

\newtheorem{theorem}{Theorem}[section]
\crefname{theorem}{Theorem}{Theorems}
\newtheorem{lem}[theorem]{Lemma}
\crefname{lem}{Lemma}{Lemmas}
\newtheorem{proposition}[theorem]{Proposition}
\crefname{prop}{Proposition}{Propositions}
\newtheorem{cor}[theorem]{Corollary}
\crefname{cor}{Corollary}{Corollaries}
\newtheorem{qu}[theorem]{Question}
\crefname{qu}{Question}{Questions}

\crefname{conj}{Conjecture}{Conjectures}

\crefname{question}{Question}{Questions}

\theoremstyle{definition}
\newtheorem{dfn}[theorem]{Definition}
\crefname{dfn}{Definition}{Definitions}
\newtheorem{ex}[theorem]{Example}
\crefname{dfn}{Example}{Examples}
\newtheorem{rmk}[theorem]{Remark}
\crefname{rmk}{Remark}{Remarks}
\newtheorem{ntn}[theorem]{Notation}
\crefname{ntn}{Notation}{Notations}

\crefname{claim}{Claim}{Claims}

\numberwithin{theorem}{section}
\allowdisplaybreaks

\newenvironment{theorem_no_number}[1][]{\begin{trivlist}
\item[\hskip \labelsep {\bfseries Theorem \def\temp{#1}\ifx\temp\empty  #1\else  #1\fi
.}] \itshape}  {\end{trivlist}}
\newenvironment{cor_no_number}[1][]{\begin{trivlist}
\item[\hskip \labelsep {\bfseries Corollary \def\temp{#1}\ifx\temp\empty  #1\else  #1\fi
.}] \itshape}  {\end{trivlist}}
\numberwithin{theorem}{section}

\DeclareMathOperator{\im}{im}

\DeclareMathOperator{\Aut}{Aut}
\DeclareMathOperator{\Out}{Out}
\DeclareMathOperator{\expsum}{expsum}

\newcommand{\Z}{\mathbb{Z}}

\newcommand{\mc}[0]{\mathcal}
\newcommand{\ab}{\textnormal{\textbf{ab}}}
\newcommand{\DPAB}[1]{\mc{DP}(#1, \ab)}
\newcommand{\DP}[1]{\mc{DP}(#1)}
\newcommand{\DPRat}[1]{\mc{DP}(#1, \operatorname{Rat})}
\newcommand{\DPRec}[1]{\mc{DP}(#1, \operatorname{Rec})}

\title{The Diophantine Problem with constraints in some free-by-cyclic groups}
\author{Laura Ciobanu, Robert D. Gray and Alex Levine}
\address{Laura Ciobanu:
Technische Universit\"at Berlin, Institut f\"ur Mathematik, %Stra\ss e des 17. Juni 136, 
10623 Berlin, Germany
\and Heriot-Watt University, Department of Mathematics, EH14 4AS Edinburgh, UK}
\address{Robert D. Gray: School of Engineering, Mathematics and Physics, University of East Anglia, Norwich NR4 7TJ, UK} 
\address{Alex Levine: School of Engineering, Mathematics and Physics, University of East Anglia, Norwich NR4 7TJ, UK} 
\begin{document}
\begin{abstract}
    In this paper we explore the Diophantine Problem, that is, the solvability of equations, in one-relator and free-by-cyclic groups, showing that this is decidable in some cases and undecidable otherwise, depending on their abelianisation and centre. 

    In particular, we focus on a version of the Diophantine Problem which asks about the existence of solutions satisfying abelianisation constraints, and prove a number of results concerning the decidability of this problem for free-by-cyclic groups with centre, tree-GBS groups, cyclically pinched groups and the first hydra group.
\end{abstract}

\maketitle

\section{Introduction}
The question of decidability of (systems of) equations in groups is widely known as the {\em Diophantine Problem} for groups, and we denote it by $\mc{DP}$, or by $\mc{DP}(G)$ when it refers to a group $G$. We extend the $\mc{DP}$ by requiring that the solution sets satisfy certain constraints, and study this augmented problem in one-relator groups and more generally, free-by-cyclic groups. We give the most general statements if the groups have non-trivial centre, but we also obtain results for several groups with trivial centre. 

Among several types of constraints (see Section \ref{DP:constraints}), we focus on abelian predicates, which we also call \emph{abelianisation constraints}, and denote by $\ab$. These constraints require that the `abelian form' of the solutions, where any two symbols commute, satisfies a given set of equations as well; equivalently, the constraints can be viewed as equations in the abelianisation of $G$. This problem, denoted by $\DPAB{G}$ for a group \(G\), was introduced in \cite{CiobanuGarreta}, where it was shown to be undecidable in hyperbolic groups and right-angled Artin groups, if these have sufficiently large abelianisation; that is, they have abelianisation with torsion-free rank greater than $2$. 

Studying the Diophantine Problem with abelianisation constraints can often give
methods for studying the `plain' Diophantine Problem. For example, if the
commutator width in a group \(G\) is finite, then the commutator subgroup in
\(G\) is definable in \(\DP{G}\), and so \(\DP{G}\) and \(\DPAB{G}\) are
equivalent. This has been leveraged in Thompson's group \(F\) to prove
undecidability of the Diophantine Problem \cite{ElliottLevine}, using the
non-trivial fact that \(F\) has finite commutator width \cite{F-comm-width}. In
other cases, the commutator subgroup may not be definable in \(\DP{G}\), but one
can often use similar techniques to compare exponent sums of various elements,
as has been successful when showing the Diophantine Problems in various groups
are undecidable, including certain GBS groups \cite{BobAlexinprep}, groups of
the form \(F_3 \rtimes F_2\) \cite{BobAlexinprep} and one-relator groups
\cite{BobAlexBG}. An important open question in the area is also whether the
Diophantine Problem is decidable in the Grigorchuck group $\mathcal{G}$ and
related branch groups. Since $\mathcal{G}$ has finite commutator width \cite{BartholdiGrothLysenok}, \(\DPAB{\mathcal{G}}=\DP{\mathcal{G}}\), and so
understanding the Diophantine Problem with abelianisation constraints would settle this open
question.

When free-by-cyclic groups have non-trivial centre, the fact that the Diophantine Problem is decidable follows from \cite{equations_VDP} (see Theorem \ref{thm:DP}). The main goal of the paper is to show that a stronger result holds in these groups when their abelianisation is virtually $\Z$: the Diophantine Problem with abelianisation constraints is decidable. 

We use several characterisations of one-relator groups with centre and the related classes of groups described below. In particular, we will use the fact that all groups of type (ii) - (v) contain (i) (see Section \ref{seq:equivalences} for a justification of the implications): 
\begin{itemize}
\item[(i)] infinite one-relator groups with non-trivial centre (O);
\item[(ii)] generalised Baumslag-Solitar (GBS) groups with non-trivial centre (GBS);   
\item[(iii)] [finitely generated free]-by-cyclic groups with non-trivial centre (FbC);    

\item[(iv)] free-by-cyclic groups of the form $F \rtimes_\phi \Z$ for some finitely generated free group $F$ and some automorphism $\phi \in \Aut(F)$ that has finite order in $\Out(F)$ (Fb$\Z$f);
\item[(v)] virtually $F \times \Z$ for some finitely generated free group $F$ (vF$\Z$). 
\end{itemize}
$$ (i) \xRightarrow[]{\textrm{\cite{BaumslagTaylor68}}} (ii) \overset{\textrm{\cite{Lev15}}} \iff (iii) \overset{\textrm{\cite{Lev15}}} \iff (iv) \xRightarrow[]{Lemma ~ \ref{lem:finorder}} (v)$$

The main result of \cite{equations_VDP}, which states that the Diophantine Problem is
decidable in all groups that are virtually direct products of hyperbolic groups,
implies the following. 
\begin{theorem}\label{thm:DP}
    The Diophantine Problem is decidable in (v), and therefore in (i) -- (iv).
\end{theorem}
Furthermore, the groups we are interested in (one-relator with centre, free-by-cyclic, GBS) are finite-index subgroups of direct products of specific hyperbolic groups (see Prop. \ref{fbyc-embed}), in the spirit of \cite{equations_VDP}, and the Diophantine Problem is well understood in these direct products. Thus we can reduce the Diophantine Problem (with or without constraints) in free-by-cyclic and one-relator groups with centre to the problem of solving equations in direct products of hyperbolic groups with recognisable (plus some additional) constraints. Moreover, when \(\DPAB{G}\) is decidable, we can describe the solutions sets to systems of equations in terms of formal languages, as EDT0L languages. This is a class of languages that describes the solutions in many classes of groups, in particular hyperbolic \cite{CE21} and abelian, which are relevant here (see \cite{CL25} for an overview). 

Our first main result is:

\begin{theorem}
        \label{fbyc-centre-decidable}
       Let \(G\) be a [finitely-generated free]-by-cyclic group with centre and virtually cyclic
       abelianisation. Then \(\DPAB{G}\) is decidable and EDT0L.
    \end{theorem}

    This immediately implies:
\begin{cor}
        \label{orgwc-decidable-cor}
       Let \(G\) be a one-relator group with centre and virtually cyclic
       abelianisation. Then \(\DPAB{G}\) is decidable and EDT0L. 
    \end{cor}
The class of groups in the statement of  \cref{fbyc-centre-decidable} are in fact exactly the tree-GBS groups, that is, the generalised Baumslag-Solitar groups whose defining graph (as a graph of groups) is a tree. Thus we have 
\begin{cor}
Let $G$ be a tree-GBS group. 
Then \(\DPAB{G}\) is decidable and EDT0L. 
\end{cor}
In contrast in recent work it has been shown that there are GBS-groups with undecidable Diophantine Problem; see \cite{BobAlexinprep}. 
Using this connection we shall see that \cref{fbyc-centre-decidable} may also be applied to the more general class of tree-products of infinite cyclic groups (i.e. a graph product where all vertex groups are $\mathbb{Z}$ and the defining graph is a tree) to show the following. 
\begin{theorem}\label{thm:TreeProduct} Let $G$ be a tree product of infinite cyclic groups.  Then $G$ has decidable Diophantine Problem with EDT0L solutions.  Furthermore:  
\begin{enumerate} 
\item if all edge groups are infinite cyclic then \(\DPAB{G}\) is decidable and EDT0L; 
\item otherwise, \(\DPAB{G}\) is undecidable. 
  \end{enumerate}   
\end{theorem}    

When the torsion-free rank of the abelianisation is $\geq 3$, we obtain undecidability results.

\begin{theorem}
    \label{fbyc-big-ab}
    Let \(G\) be a [finitely generated free]-by-cyclic group with non-trivial centre and a torsion-free
    rank \(3\) or greater abelianisation. Then \(\DPAB{G}\) is undecidable.
\end{theorem}

% \begin{proposition}
%        \label{big-ab-if-stuff-defable}
%        Let \(G = F \rtimes \Z\) where \(F\) is finitely generated free. If
%        there exist \(h_1, h_2 \in F\) such that \(\langle h_1 \rangle\) and
%        \(\langle h_2 \rangle\) are \(\DPAB{G}\)-definable and \(\ab_G(\langle
%        h_1, h_2 \rangle) \cong \Z^2\), then \(\DPAB{G}\) is undecidable.
%   \end{proposition}

%% Comment about finite index 
We note that while we do utilise finite index subgroups and extensions in our reasoning in this paper, in general it is unknown whether decidability of the Diophantine Problem is preserved under taking finite index subgroups or extensions. Also, \(\DPAB{G}\) is know not to be preserved when taking finite index subgroups. 
It is known that for every RAAG, there is a corresponding RACG that contains it as a subgroup of finite index.
Now in \cite{CiobanuGarreta} it is proved that 
(a) If $G$ is a right-angled Artin group that is not abelian then 
\(\DPAB{G}\) is undecidable, and also that 
(b) If $H$ is a right-angled Coxeter group then \(\DPAB{H}\) is decidable. Hence if we take any non abelian RAAG $G$ it has a finite index extension $H$ with $G \leq H$ of finite index (the RACG described above) such that \(\DPAB{H}\) is decidable  
but \(\DPAB{G}\) has undecidable with $G$ a finite index subgroup of $H$.     
%%In particular if we take the RACG $K$ with graph the path $P_3$ with three vertices then \(\DPAB{K}\) is decidable even though $K$ embeds a subgroup $\overline{K}$ of finite index where $\overline{K}$ is isomorphic to the RAAG on $P_3$, that is $\overline{K} \cong F_2 \times \mathbb{Z}$, and \(\DPAB{\overline{K}}\) is undecidable. 

    For groups with trivial centre we have the following results.

    \begin{theorem}
    Let \(G = \langle a, b, t \mid tat^{-1} = ab, tbt^{-1} = b \rangle\) be the `first hydra group'. Then
    \(\DPAB{G}\) is undecidable
\end{theorem}
We note that whether the first hydra group (or any other hydra group) has a decidable Diophantine Problem is an open problem.
%%
%%, and it is also open for the other hydra groups.  
%%{\color{blue} Bob: Would it be worth mentioning that the ordinary DP remains open for the first hydra group (and I suppose it is open for all the hydra groups in fact)?}. 

Recall that a \emph{cyclically pinched
group $G$} is a free product of two finite rank free groups amalgamated over an
infinite cyclic group; that is, 
\[G=F(a_1, \dots, a_n)\ast_{\mathbb{Z}} F(b_1,
\dots, b_m),
\]
and the amalgamation is of the form $\langle U \rangle=\langle V
\rangle \cong \mathbb{Z}$, where $U \in F(a_1, \dots, a_n)$ and $V\in F(b_1,
\dots, b_m)$ are non-trivial. By noticing that these groups are relatively hyperbolic with sufficiently well-behaved parabolics, we are able to apply existing
combination results to obtain the following:
\begin{cor}
\label{DP-cyc-pinch}
The existential theory (and hence the Diophantine Problem) is decidable in all cyclically pinched groups.
\end{cor}

Cyclically pinched groups provide a rare example of a class of groups where
the decidability of the Diophantine Problem with abelianisation constraints
can be fully classified, without a `gap' corresponding to a specific
abelianisation. In this case, those with virtually cyclic abelianisation
are amongst the decidable examples.

 \begin{theorem}\label{thm:cycpinDPAB}
        Let \(G\) be a cyclically pinched group. Then \(\DPAB{G}\) is decidable
        if and only if \(G\) has exactly two generators. In the
        case that \(G\) has two generators, \(\DPAB{G}\) is EDT0L.
    \end{theorem}

%%%%%%%%%%%%%%%%%%%%%%%%%%%%%%
                                                                                                %%%

% Expanding on this a bit...
In \cite{CiobanuGarreta} it is shown that for a hyperbolic group $G$ if the abelianisation of $G$ has torsion-free rank $\geq 2$  $\DPAB{G}$ is undecidable, while if the abelianisation of $G$ is finite then $\DPAB{G}$ is decidable. This motivates the study of $\DPAB{G}$ for hyperbolic groups (and finitely generated groups more generally) with torsion-free rank 1. As explained in \cite{CiobanuGarreta} the decidability of  \(\DPAB{G}\)  when $G$ has decidable Diophantine Problem and its abelianisation has torsion-free rank 1 remains an open problem. Motivated by this in  Problem 4 in \cite{CiobanuGarreta} it is asked whether 
the braid group on three strands  $\langle a,b \mid aba= bab \rangle$ (or more generally, odd dihedral Artin groups) have decidable $\DPAB{G}$. Observing that every odd dihedral Artin groups are isomorphic to a torus knot group of the form $\langle x,y \mid x^2 = y^m \rangle$ (see for example \cite{Crowe25}*{Lemma 2.8}) we may apply the results in this paper to give a positive answer to that question: 
   \begin{cor} Let $T_{p,q}=\langle a,b \mid a^p=b^q \rangle$ with $p,q \geq 2$  Then \(\DPAB{G}\) is decidable. In particular \(\DPAB{G}\) is decidable in every torus-knot group and $\DPAB{G}$ is decidable in every odd dihedral Artin group. 
    \end{cor}

%In fact our results give several ways of seeing that \(\DPAB{G}\) is decidable in the groups $T_{p,q}=\langle a,b \mid a^p=b^q \rangle$ since these one-relator groups are both (a) one-relator groups with non-trivial centre and also (b) cyclically pinched one-relator groups.
%{\color{purple} AL: Technically not - we obtain the positive results for cyclically pinched groups by noticing that they are also one-relator with centre (that is, precisely the torus knot case)}

The main tool used to prove undecidability results in this paper is \emph{interpretability} using systems of equations, which allows us to translate one structure into another and to reduce the Diophantine Problem from one structure to the other. The main reductions are to the Diophantine Problem in the ring of integers, which is a classical undecidable problem (Hilbert's 10th problem). 

\subsection*{Connections to computer science.} The Diophantine Problem with
abelianisation constraints in groups is inspired by work in theoretical computer
science, where the study of equations with language theoretic, length and
abelianisation constraints in free monoids has a long tradition and is a very
active field of research \cites{Durnev74, Abdulla, RichardBuchi1988, DayManeaWE,
GarretaGray}. 

Deciding whether a free monoid equation has solutions satisfying linear length constraints is a major open question, of both theoretical importance (if undecidable, it would provide a new solution to Hilbert's 10th problem for Diophantine polynomials) and practical implications, in the context of string/SMT solvers for security analysis, formal verification, the robustness of Web applications, and more \cite{Amadini}. 
  SMT solvers, which have boomed over the last 10 years, consider first-order formulae in a wide context, such as the real numbers, integers, data structures, and strings over finite alphabets. The question is then whether there is a solution for a first-order formula that is significantly more complex than what is handled by SAT-solvers: the name `satisfiability modulo theories' (SMT) expresses the fact that the solvers work `modulo' a certain theory (beyond the Boolean framework) in first-order logic with equality. In algebraic terms, this means solving equations in free monoids and related structures (as many of the constraints lead naturally to commutativity or other relations between the alphabet letters), as every string can be seen as an element over the generators of a finitely generated monoid or group.

%----------------------
% 
%- We say we prove a result that ``partially answers Problem 4 in \cite{CiobanuGarreta}''. I think it might be good to expand on this e.g. say exactly what the problem is that we answer. Quoting from that paper it says:
%
%\noindent ``There are other intriguing groups where the abelianisation has torsion-free rank=1, such as the Baumslag–Solitar groups $BS(1,n)$, or the braid group on three strands $\langle a,b | aba= bab \rangle$ (or more generally, odd dihedral Artin groups).  In all these cases, DP is known to be decidable, but (DP,ab) remains open."
%
%Another related quote from the paper that is relevant is the following: 
%
%``The decidability of  \(\DPAB{G}\)  when $G$ has decidable Diophantine Problem and its abelianisation has torsion-free rank 1 remains an open problem. This applies, for example, to hyperbolic groups with first Betti number equal to 1, or to certain classes of one-relator groups, including dihedral Artin groups and soluble Baumslag-Solitar groups (see Section 8)."
%                                                                                                %%%
%%%%%%%%%%%%%%%%%%%%%%%%%%%%%%%
%
%A consequence of both Corollary \ref{orgwc-decidable-cor} and Theorem \ref{thm:cycpinDPAB} is the following, which partially answers Problem 4 in \cite{CiobanuGarreta}. 
%

 \section*{Declaration of (no) AI usage}

 No AI tools were used for proving any results, making any conjectures or generating any examples. No AI tools were used for any aspects of the writing or editing of this paper. 
 
    \section*{Acknowledgements}
    
    The first author benefited from the support provided by the Deutsche Forschungsgemeinschaft (DFG, German Research Foundation) under Germany's Excellence Strategy -- The Berlin Mathematics Research Center MATH+ (EXC-2046/1, EXC-2046/2, project ID: 390685689).

    The research of the second and third-named authors was supported by the EPSRC Fellowship grant EP/V032003/1 ‘Algorithmic, topological and geometric aspects of infinite groups, monoids and inverse semigroups’.

\section{Preliminaries}

Let $\mathbf{x}=\{X_1, \dots, X_m\}$ be a set of variables, where $m\geq 1$.
% For any group $G$ and finite set of variables $\mathbf{x}$, an \emph{equation} with
%coefficients $g_1,\dots, g_{m+1}$ from $G$ is a formal expression
%$g_1X_1^{\epsilon_1}g_2 X_2^{\epsilon_2}\dots X_m^{\epsilon_m}g_{m+1}=1_G,$
%where $\epsilon_j = \pm 1$ for all $1 \leq j \leq m$,
%and $X_j \in \mathbf{x}.$ 
%Such an equation is called \textit{satisfiable} if there exist values for
%the $X_j$'s in $G$ with which the above identity in $G$ is satisfied;
%each such set of values for the $X_j$ is a \textit{solution}.   
%A finite set of equations with coefficients in $G$
%is a \textit{system of equations} over $G$, and is satisfiable if there are assignments to the $X_j$
%such that all of the equations in the system are satisfied.
For a group $G$, a \textit{finite system of equations} in $G$ over the variables $\mathbf{x}$ is a
		finite subset $\mathcal{E}$ of the free product $G \ast F(\mathcal{\mathbf{x}})$, where
		\(F(\mathbf{x})\) is the free group on \(\mathbf{x}\). If
		\(\mathcal{E} = \{w_1, \ \ldots, \ w_n\}\), then a \textit{solution} to the
		system \((w_1 =1) \wedge \cdots \wedge (w_n = 1)\) is a homomorphism \(\phi \colon G \ast F(\mathbf{x}) \to
		G\), such that \(\phi(w_1) = \cdots =\phi (w_n) = 1_G\) and
		 \(\phi(g) = g\) for all \(g \in G\). If $\mathcal{E}$ has a solution, then it is \emph{satisfiable}.
		 
%	\begin{example} Consider the system $\mathcal{E}=\{w_1, w_2\}\subset F(a,b) \ast F(X_1, X_2)$ over the free group $F(a,b)$, where $w_1=X_1^2(abab)^{-1}$, $w_2=X_2X_1X_2^{-1}X_1^{-1}$; we set $w_1=w_2=1$, which can be written as $X_1^2=abab, X_2X_1=X_2X_1$. The solutions are $\phi(X_1)=ab, \phi(X_2)=(ab)^k$, $k \in \mbb{Z}$.
	%\end{example}

For a group $G$, we say that systems of equations over $G$
are \emph{decidable} over $G$ if there is an algorithm to determine whether any
given system is satisfiable.  The question of decidability of systems of equations in \(G\) is called the
{\em Diophantine Problem} for $G$, and denoted $\DP{G}$.

\subsection{The Diophantine Problem with constraints}\label{DP:constraints}

We can augment the Diophantine Problem for a group by requiring that the solutions belong to certain sets, which we call \emph{constraints} (see \cite{CiobanuGarreta} and \cite{CEL25} for more background). We start by defining formal language constraints. 

Let $\Sigma$ be a finite alphabet and let $S=\Sigma^{\pm 1}$. Suppose $G$ is a group generated by $S$, and let $\pi \colon  S^\ast \rightarrow G$ be the natural projection from the free monoid $S^*$ generated by $S$ to $G$, taking a word over the generators to the element it represents in the group. A language is \emph{regular} if it is recognised by a finite state automaton, as is standard.
%When $w$ is a word over $S$, and we write $|w|$ to denote the length of $w$.

\begin{dfn}\label{def:rat_reg_con}
\ \begin{enumerate}
\item[(1)] 
A subset $L$ of $G$ is {\em recognisable} if the full preimage
$\pi^{-1}(L)$ is a regular subset of $S^*$.
\item[(2)] A subset $L$ of $G$ is {\em rational} if $L$ is the
image $\pi(R)$  of a regular subset $R$ of $S^*$.
%\item[(3)] A regular subset $L'$ of $S^*$ is
%\emph{quasi-isometrically embedded} (q.i. embedded) in $G$ if there exist
%$\lambda \geq 1$ and $\mu \geq 0$ such that, for any $w \in L'$,
%$|\pi(w)|_G \geq \frac{1}{\lambda}|w|-\mu.$
%\item[(4)] A rational subset $L$ of $G$ is \emph{quasi-isometrically
%embeddable} (q.i. embeddable) in $G$ if there exists a quasi-isometrically embedded
%regular subset  $L'$ of $S^*$ such that $\pi(L')=L.$
\end{enumerate}
\end{dfn}
It follows immediately from the definition that recognisable subsets of $G$ are
rational.

Let $\mathcal E$ be a system of equations on variables $\mathbf{x}=\{X_1, \dots,
X_k\}$ in a group $G$. The \emph{Diophantine Problem with rational (resp.
recognisable) constraints} asks about the existence of
solutions to $\mathcal E$, with some of the variables restricted to taking
values in finitely many specified rational (resp. recognisable) sets. We write \(\DPRat{G}\) and \(\DPRec{G}\) for the two cases. 

While there are many non-rational constraints of interest (see \cite{CioDLT24}),
we focus here on abelianisation constraints.  For any group $G$, let $\ab\colon G\to
G^{\ab}$ be the natural abelianisation map to $G^{\ab}= G/G'$, that is, the
quotient of $G$ by its commutator subgroup. We write $\mc{DP}(G, \ab)$ for the
$\mc{DP}$ where an abelian predicate is added, or equivalently,
\emph{abelianisation constraints} are imposed. A set of abelianisation
constraints is a system $\Theta$ of equations in the group $G^{\ab}$, and
$\mc{DP}(G, \ab)$ asks whether a solution to $\mathcal E$ exists such that the
abelianisation of the solution satisfies the system $\Theta$ in $G^{\ab}$. 
%We
%will also consider, more generally, the \emph{Diophantine Problem with
%\(E\)-constraints} for some fixed subset \(E \subseteq G\), which allows
%restricting variables to the set \(E\). We denote this by
%\(\DPC{G}{E}\), and note that \(\DPAB{G}\) and \(\DPC{G}{[G, G]}\) are the same. 

\begin{ex}

Let $F(A)$ be the free group of rank $2$ with basis $A=\{\alpha, \beta\}$ and abelianisation $\ab \colon F(A) \mapsto \Z^2$. For the set of variables $\{X, Y, Z\}$ consider the equation 
%\begin{equation}\label{eqn}
$$X^2\alpha Y=YZ^{-2}\alpha.$$
%\end{equation}

An instance of $\mc{DP}(G, \ab)$ would be to ask if there is a solution $\sigma$ such that $$\ab(\sigma(X)) + \ab(\sigma(Z))=(0,0).$$
The answer is `yes', as the triple $(\alpha\beta, \alpha, \alpha^{-1}\beta^{-1})$ shows: $\sigma_0(X)=\alpha \beta$, $\sigma_0(Y)=\alpha$ and $\sigma_0(Z)=(\beta\alpha)^{-1}$ satisfies $\sigma_0(X^2\alpha Y)=\sigma_0(YZ^{-2}\alpha)$.
Moreover, $\sigma_0$ satisfies $\ab(\sigma_0(X))=-\ab(\sigma_0(Z))=(1,1)$, so the abelianisation constraint holds.
\end{ex}

A key notion associated to a system of equations is that of a $\DP{G}$-definable set.

\begin{dfn}
%{\color{blue} AL: We never actually use the finite unions bit - just definable
%with systems on the nose. Worth updating?}

\label{d: e-def} 
Let $G$ be a group.
A set $D \subset G\times \overset{m}{\ldots} \times G = G^m$ is called
%\emph{definable by positive existential formulas} in $G$, or
$\DP{G}$-\emph{definable}, if there exists a system
$\mc{E}_{D}$ of equations with variables $(X_1, \dots, X_m,
Y_1, \dots, Y_k)=(\mathbf{x}, \mathbf{y})$, and constants from $G$, satisfying
the following: a tuple $\mathbf{a}\in G^m$ is in $D$ if and only if  there
exists \(\mathbf{b} \in G^k\) such that \((\mathbf{a}, \mathbf{b})\) is
a solution to the system $\mc{E}_{D}$.
  In this case $\mc{E}_{D}$ is said to \emph{define} $D$ in $G$.
\end{dfn}

\begin{ex}\label{ex:central}

\ \begin{itemize}
  
\item[(1)] The most frequent type of $\DP{G}$-definable set in a group $G$ encountered in this paper will be the \emph{centraliser} of an element $g$, that is, the set $C_G(g) = \{h\in G\mid gh=hg\}$. An element $x\in G$ belongs to $C_G(g)$ if and only if $x$ is a solution to the equation $Xg=gX$.

\item[(2)] Another classical example is the \emph{centre} $Z(G)=\{g\in G\mid gh=hg \ \forall h\in G\}$ of a group $G$ generated by a finite set $\{g_1, \dots, g_n\}$. Indeed, an element $x \in G$ lies in  $Z(G)$ if and only if $x$ satisfies the system of equations $([X,g_1]=1) \wedge \cdots \wedge ([X,g_n]=1)$.
\end{itemize}
\end{ex}

Using $\DP{G}$-definability one can interpret one algebraic structure into
another in the standard model-theoretic way. We do not formalise this here, but
refer the reader to \cite{CiobanuGarreta}*{Section 2.3} for the precise
definitions.  A common thread throughout this paper is to use interpretability
(see \cite{CiobanuGarreta}) to encode addition and multiplication in $\Z$ into
the $\mc{DP}(G, \ab)$ for some group $G$, and by the undecidability of Hilbert's
10th Problem, get the same result for $G$.

%some of which are restricted to lying in specified rational or recognisable sets, respectively.
% We write $\mc{DP}(G, \ab)$ for the $\mc{DP}$ where an abelian predicate is added, or equivalently, \emph{abelianisation constraints} are imposed. A set of abelianisation constraints is a system $\Theta$ of equations in the group $G^{\ab}$, and $\mc{DP}(G, \ab)$ asks whether a solution to $\mathcal E$ exists such that the abelianisation of the solution satisfies the system $\Theta$ in $G^{\ab}$.
%
%That is, we ask whether there is a solution $\phi$ of $\mathcal E$ satisfying a systems of equations in the abelianisation of $\Sigma^*$ or $F(\Sigma)$, such as $a_1 \ab(\phi(X_1))+ \dots + a_k \ab(\phi(X_k))=c.$

\subsection{Existential theory} For various combination results, we will use the existential theory
in a group \(G\). The
existential theory is a generalisation of the Diophantine Problem, to allow all
first-order sentences built from existential quantifiers, the constants in \(G\)
together with the logical operators $\wedge$ (`and'), $\vee$ (`or') and $\lnot$ (`not'). De Morgan's laws allow one
to factor all of the or operators to the front of the sentence, so all such
sentences are equivalent to finite disjunctions of systems of equations and
inequations in our group (where by an inequation, we mean an expression of the
form \(w \neq 1\) where \(w \in F(\mathbf{x}) \ast G\)).

\subsection{EDT0L solutions} 
A fruitful line of research over the last ten years has been to describe the
solutions to systems of equations in groups in terms of formal languages. The
class of EDT0L languages has proved particularly suitable, as it describes
solutions in many important classes of groups, including hyperbolic groups
\cite{CE21}, virtually direct products of hyperbolic groups \cite{Levine23} and
sinlge equations in one variable in the Heisenberg group \cite{LevineHeis}. We
do not need the definition of this class of languages here as we rely on
combination results from known characterisations, and refer the reader to
\cite{CL25} for definitions and an overview of the area.

When passing to direct products, the decidability of the Diophantine Problem
with rational constraints is not closed, with \(F_2 \times \Z\) being a
counterexample \cite{BobAlexinprep} and \(F_2 \times F_2\) even having an
undecidable subgroup membership problem \cite{Mihaelovagenwp}. If the rational
constraints are forced to be fully contained in one of the factors, then one
can obtain positive results in this direction.

 \begin{proposition}
    \label{dir-prod-rat-constraints}
        Let \(G_1\) and \(G_2\) be groups. Suppose that \(\DPRat{G_1}\) and \(\DP{G_2}\) are decidable
        (EDT0L).
        Then the Diophantine Problem
        with rational constraints fully contained in \(G_1 \times \{1\}\) is
        decidable (EDT0L) in \(G_1 \times G_2\).
    \end{proposition}

    \begin{proof}
        Let \(\mathcal{E}\) be a system of equations in \(G_1 \times G_2\),
        with set of solutions \(\mathcal S\). Let \(\mathcal E_1\) and
        \(\mathcal E_2\) be the projections of \(\mathcal{E}\) onto \(G_1\)
        and \(G_2\), respectively, with sets of solutions \(\mathcal S_1\)
        and \(\mathcal S_2\), respectively. Then
        \[
            \mathcal S = \{(x, y) \mid x \in \mathcal S_1, y \in \mathcal{S}_2\}.
        \]
        So if we add rational constraints to \(\mathcal E\) that are fully
        contained in \(G_1\), under the projections these just add rational
        constraints to \(\mathcal E_1\) and leave \(\mathcal E_2\) unchanged,
        and the above equality still holds. In particular, \(\mathcal S\)
        is empty if and only if \(\mathcal{S}_1 = \varnothing\) or
        \(\mathcal S_2 = \varnothing\), which we can decide using
        \(\DPRat{G_1}\) and \(\DP{G_2}\), respectively.

        For the claim about EDT0L solutions, we can apply
        \cite{Levine23}*{Proposition 4.10}. This does not cover the rational
        constraints, but it does show that if \(\mathcal S_1\) and \(\mathcal
        S_2\) correspond to EDT0L solution languages, then so does \(\mathcal
        S\), with respect to the natural inherited normal form of direct
        products. Thus even with the addition of rational constraints to
        \(\mathcal S_1\), the proof of \cite{Levine23}*{Proposition 4.10} will
        still imply that \(\mathcal S\) corresponds to an EDT0L solution
        language.
%
%        {\color{red} If the above paragraph is not okay, we can just copy
%        and paste the proof of \cite{Levine23}*{Proposition 4.10} and
%        add `rational constraints' to one of the solution languages. This
%        is about half a page, as it relies heavily on Lemma 3.5 in the
%        same paper.}
    \end{proof}

\subsection{Same groups, different looks}\label{seq:equivalences}
The focus of this paper are free-by-cyclic groups, so semidirect products of the form $F_n \rtimes_\phi \Z$, where $\phi$ is an automorphism of $F_n$, with $n\geq 1$.

In particular, we consider free-by-cyclic groups with non-trivial centre. Baumslag and Taylor \cite{BaumslagTaylor68} proved that each one-relator
group with non-trivial centre is $\{$ f.g. free $\}$-by-cyclic $F_n
\rtimes_\phi \Z$ for some $n\geq 1$. They also proved that there is an algorithm that takes any
one-relator group as input and decides whether that group has a non-trivial
centre. 

Another relevant class of groups are the   
Generalised Baumslag-Solitar (GBS) groups. These are the fundamental groups of
finite graphs of infinite cyclic groups \cite{Lev15}. We assume familiarity with
Bass-Serre theory and define each such group based on a graph of groups $\Gamma
(V,E)$, where $\Gamma$ is a connected graph with vertex set $V=V(\Gamma)$ and
edge set $E=E(\Gamma)$, and an infinite cyclic group is attached to every vertex
and edge. To every edge $e \in E$ we associate a start-vertex $s(e)$ and
terminus $t(e)$, an involution $e\mapsto \overline{e}$, and local groups
$G_{e}=G_{\bar{e}}$ and $G_{s(e)}$, $G_{t(e)}$ (all isomorphic to $\Z$).
Furthermore, there are inclusion maps $\iota_{e,1} \colon G_e \mapsto G_{s(e)}$ and
$\iota_{e,2} \colon G_e \mapsto G_{t(e)}$ given by multiplication by some integers
$p_e$ and $q_e$. 

It turns out that GBS groups with non-trivial centre are free-by-cyclic:

\begin{proposition}[{\cite{Lev15}*{Proposition 4.1}}]
For any group $G$ the following are equivalent: 
\begin{enumerate}
\item $G$ is a GBS group with a non-trivial centre;   
\item there exists $n\geq 1$ and $\phi \in \mathrm{Out}(F_n)$ of finite order such that $G = F_n \rtimes_\phi \Z$.    
\end{enumerate}
\end{proposition}

In fact, it follows from \cite{Lev15} that the classes of groups (ii), (iii) amd (iv) listed in the Introduction are all equal to each other. %The equivalence of the second and third points above is discussed for instance in the introduction to the paper of Bridson and Piwek \cite{BridsonPiwek}.  
\begin{itemize}

\item[(ii)] generalised Baumslag-Solitar (GBS) groups with non-trivial centre (GBS);   
\item[(iii)] [f.g. free]-by-cyclic groups with non-trivial centre (FbC);    

\item[(iv)] free-by-cyclic groups of the form $F \rtimes_\phi \Z$ for some finitely generated free group $F$ and some automorphism $\phi \in \Aut(F)$ that has finite order in $\Out(F)$ (Fb$\Z$f);
%\item GBS groups with non-trivial centre;   
%\item $\{$ f.g. free $\}$-by-cyclic groups with non-trivial centre;    
%\item free-by-cyclic groups where $F \rtimes_\phi \Z$ for some finitely generated free group $F$ and some automorphism $\phi \in \Aut(F)$ that has finite order in $\Out(F)$. 
\end{itemize}

Thus if $G$ is a GBS group with a non-trivial centre then $G$ is virtually $F_r \times \Z$, $r\geq 1,$ by Lemma \ref{lem:finorder}, which is a standard result. Therefore free-by-cyclic groups with acting autmorphism $\phi$ of finite order in $Out(-)$ (Fb$\Z$f) are virtually direct products, where the product has the form $F \times \mathbb{Z}$ (vF$\Z$).

%{\color{red} AL: Is the following not citable?}
\begin{lem}\label{lem:finorder}
Let $G$ be any group and let $\phi$ be an automorphism of finite order in $Out(G)$. The group $G \rtimes_{\phi}\mathbb{Z}$ contains $G \times \mathbb{Z}$ with finite index.
\end{lem}

%\begin{proof}
%Suppose $\phi^k=\tau_h$ for some integer $k\geq 1$ and an inner automorphism $\tau_h$. Then write $G \rtimes_{\phi}\mathbb{Z}=\langle G, t \mid t a_it^{-1}=\phi(a_i) \rangle$, where $a_i$ are the generators of $G$.
 
 %Consider the subgroup $H=\langle a_1, \dots, a_n, t^k \rangle$ of $G$.  From $t a_it^{-1}=\phi(a_i)$ we get that $t^k a_i t^{-k}=\phi(a_i)^k=ha_i h^{-1}$, so $(h^{-1} t^k)a_i (h^{-1} t^k)^{-1}=a_i$. If we write $r=h^{-1} t^k$ then $H=\langle G, r \mid r a_ir^{-1}=a_i \rangle \cong G \times \mathbb{Z}$. Since each $t^{i}, 1\leq i \leq k$ gives a coset representative of $H$ in $G$, $H$ has finite index in $G \rtimes_{\phi}\mathbb{Z}$.
 %\end{proof}

%\lc{give an example of a virtually $F \times \Z$ group that is not free-by-cyclic with centre.}

\begin{ex}
The classes of groups (ii)-(iii)-(iv) of free-by-cyclic groups with centre (FbC,
equivalently GBS or Fb$\Z$f) do not coincide with the class (v) of virtually $F
\times \Z$ for some finitely generated free group $F$, as free-by-cyclic groups
are torsion-free, whilst virtually \(F \times \Z\) groups need not be.
%he following example
%shows. The group $F \times \Z \times C_2$, where $C_2$ is the cyclic group with
%two elements, is clearly virtually $F \times \Z$, but is not free-by-cyclic,
%since free-by-cyclic groups are torsion-free.
%The group $$\langle a, b, t \mid tat^{-1}=b, tbt^{-1}=a\rangle$$
%is a free-by-cyclic group $F \rtimes_{\phi}\mathbb{Z}$ with acting automorphism $\phi$ swapping $a$ and $b$. Since the automorphism has finite order, by Lemma \ref{lem:finorder} this group is virtually $F \times \mathbb{Z}$. This group does have centre, so doesn't work. 

\end{ex}

The main protagonists of the paper have yet another characterisation. They are central extensions of hyperbolic groups (even stronger, free-by-cyclic groups with centre are central extensions of virtually free groups), and so by work of Liang have decidable Diophantine Problem \cite{Liang}.
%\begin{proposition}
 %   Free-by-cyclic groups with centre are central extensions of virtually free groups.
%\end{proposition}

%\lc{mention the Bridson-Piwek reference for this as well, or maybe instead of proof}

%\begin{proof}
  %  Let $G=F \rtimes \Z$ be (f.g.free)-by-cyclic. The centre $Z(G)$ has trivial intersection with the free normal subgroup $F$ (except in the case of the free abelian group of rank 2). So the quotient $G/{Z(G)}$ is free-by-finite.
%\end{proof}
\subsection{Computable translations between (i) -- (v)}

We have from \cite{Levitt07} that [f.g. free]-by-cyclic groups with centre
are in fact GBS with centre, which admit a presentation that we wish to work with. However,
when translating between these different isomorphic presentations, we require the process
to be effective, so that our algorithm can perform them. The proof is not difficult; we simply
iterate through all possible Tietze transformations.

\begin{lem}\label{lem:helpful2}
There is an algorithm that takes as input any [f.g. free]-by-cyclic presentation for a group $G$ with non-trivial centre and a virtually cyclic abelianisation and outputs a GBS presentation defining $G$ of the form 
        \[
            \langle \{a_v\mid v \in V(T)\} \mid \{a_{s(e)}^{p_e} = a_{t(e)}^{q_e}
            \mid e \in E(T)\} \rangle,
        \]
        for some finite rooted tree \(T\) and  some \(p_e, q_e \in \Z_{> 0}\) for all \(e \in E(T)\).
\end{lem}
\begin{proof}
Let $G$ be a free-by-cyclic group with centre and a virtually cyclic
abelianisation. By \cite{Levitt07}*{Theorem 1.1(1)} there is a finite rooted tree \(T\)
such that \(G\) admits a presentation of the form
        \[
            \langle \{a_v\mid v \in V(T)\} \mid \{a_{s(e)}^{p_e} = a_{t(e)}^{q_e}
            \mid e \in E(T)\} \rangle,
        \]
        for some \(p_e, q_e \in \Z_{> 0}\) for all \(e \in E(T)\).
satisfying the conditions in the statement of the lemma, that defines $G$.  
The set of all finite presentations of $G$ can be recursively enumerated by starting with the presentation $\langle A \; | \; r=1 \rangle$ and repeatedly applying Tietze transformations.   
So the algorithm proceeds by listing the finite presentations of $G$ in this way and at each step checking whether any of the finite presentations listed thus has the GBS form above. 
Since we know at the outset that $G$ admits some presentation of this form this algorithm will terminate at the point where the finite list of presentations of $G$ we have computed contains at least one presentation of the correct form.   
%
%obtained from the given presentation 
%$\langle A \; | \; r=1 \rangle$ via some finite sequence of Tietze transformations.  kl
%it follows that all finite presentations   
\end{proof}

The argument used in the proof of Lemma~\ref{lem:helpful2} is quite general and can be used to establish the following result. 

\begin{lem}\label{lem:helpful}
There is an algorithm that takes as input any non-cyclic one-relator group $G$ with non-trivial centre and outputs a   
[f.g. free]-by-cyclic $F_n \rtimes_\phi \Z$ presentation defining $G$, that is, a finite presentation of the form 
\begin{equation}\label{eqn:FBC}
    \langle A, t \; \mid \; tat^{-1} = w_a  \; (a \in A) \rangle 
\end{equation}
where each $w_a \in F_A$ and the map $a \mapsto w_a$ $(a \in A)$ induces an automorphism $\phi$ of the free group $F_A$.    
\end{lem}
\begin{proof}
There are classical algorithms due to Nielsen and Whitehead (see \cite{LyndonAndSchupp})
%%
%%It is a classical fact (e.g. using the \cite{}) that there is an algorithm 
%%
that take 
%%any finite set of words $W$ from the free group $F_A$, where $A$ is finite, and determines whether or not $W$ is a basis for the free group $F_A$. It follows that there is an algorithm that takes any 
any mapping $\theta \colon A \rightarrow F_A$ (where $A$ is a finite set) as
input and and determine whether or not $\theta$ extends to an automorphism of
$F_A$. This implies that there is an algorithm that takes any finite group
presentation $\langle B \; \mid \; Q \rangle$ and decides whether or not it has
the form \eqref{eqn:FBC} given in the statement of the lemma. We can thus apply
exactly the same argument from the proof of \cref{lem:helpful2}.
%all finite sequences of Tietze transformations until we come to a presentation that is
%of the desired form, which must terminate since we assumed \(G\) was non-cyclic
%one-relator with centre, and is thus free-by-cyclic with centre by \cite{BaumslagTaylor68}, and so admits such a presentation.
%Using this fact, the lemma now follows using exactly the same argument as in the proof of Lemma~\ref{lem:helpful}.   
\end{proof}

\subsection{Exponent sums}

A useful concept when dealing with abelianisations of groups is that of exponent sums. This simply
considers projections of the group onto maximal copies of \(\Z\) that are present in
the abelianisation. We typically consider exponent-sums of generators, although the
definition is more general.

\begin{dfn}
    Let \(G\) be a group and let \(a\in G\) be such that \(\ab(a) \cong \Z\) and
    \(\ab(G) = \langle \ab(a) \rangle \oplus K\), for some group \(K\); that is,
    \(\langle \ab(a) \rangle\) is a maximal copy of \(\Z\) in \(\ab(G)\). The
    \emph{exponent-sum} of \(g \in G\) with respect to \(a\) is denoted
    \(\expsum_a(g)\) and is defined to be the image of \(g\) under the
    epimorphism from \(G\) to \(\Z\) defined by composing the map \(\ab\) with the
    projection onto \(\langle \ab(a) \rangle\). 
\end{dfn}

\section{Decidability for free-by-cyclic groups with non-trivial centre and virtually cyclic abelianisation}
\label{sec:fbyc-with-centre}

%As observed above, every one-relator group with centre is actually a $\{$ f.g. free $\}$-by-cyclic group. Hence it is natural to ask to what extend the results of the previous section might be extended to the more general class of $\{$ f.g. free $\}$-by-cyclic groups with centre. Note that this class has received attention in the literature; see e.g. the recent word of Bridson and Piwek \cite{BridsonPiwek} proving profinite rigidity for this class. 

 In this section we show another facet of free-by-cyclic groups with non-trivial
 centre by embedding them into direct products, as in Proposition
 \ref{fbyc-embed}. As already observed, the class of one-relator groups with
 non-trivial centre is contained in the class of [f.g. free]-by-cyclic groups with
 non-trivial centre. 
 
 In Theorem \ref{fbyc-centre-decidable} we prove that \(\DPAB{G}\)
is decidable in a free-by-cyclic $G$ if both the centre and the abelianisation satisfy the hypothesis. The proof is similar to proving the decidability of the `plain' Diophantine Problem exhibited in \cite{equations_VDP}; one first embeds the group
as a finite-index subgroup inside a direct product of hyperbolic groups, and then
notices that this direct product has a decidable Diophantine Problem with recognisable
constraints. In our case, the hyperbolic groups are a virtually free group \(V\) and \(\Z\) and the commutator subgroup is a a finitely generated subgroup of \(V\). This setting allows us to apply \cref{dir-prod-rat-constraints}, noting that the Diophantine Problem in a virtually free group with rational constraints is decidable \cite{DahmaniGuirardel} and
has EDT0L solutions \cite{DiekertElder}, to complete the proof.

In order to prove that some of these facts occur, we first move to the world of
GBS groups based on trees. We require some notation to work with these groups.

\begin{ntn}
        If \(T\) is a rooted tree, we define an order on \(T\) (both on edges and vertices),
        where the root is minimal, and two items are comparable if they both lie on a minimal
        path from the root, with the `closer' one being smaller in this partial order.
        Further, if \(e \in E(T)\), then we define \(s(e)\) and \(t(e)\) to be the two
        ends of \(e\) such that \(s(e) < t(e)\) (note that \(s(e) \neq t(e)\) since
        \(T\) is a tree).
\end{ntn}
%\lc{Note similar recent results in https://arxiv.org/pdf/2603.17596, Lemma 4.2 and Cor. 1.5. While Ashot's results are more general our are finer and stronger, in a sense. We get more than an embedding into the direct product of virtually free and virtually abelian. Our second factor is purely abelian.}
\begin{comment}
{\color{purple} AL: I'm not sure I agree. Ashot's results are like those in
\cite{equations_VDP} - they are about groups already known to be commensurable
to direct products of free groups. Obviously, these are such examples, but
realising that is part of the point. Not that I'm trying to argue this
proposition isn't already known is some form, it follows from
\cite{equations_VDP} (or Ashot), but together with the fact that free-by-cyclic
with centre are virt. dir. prods of free groups (which we have a proof of above, but was definitely already known). Probably worth citing both of these here anyway. }
\end{comment}

    \begin{proposition}
        \label{fbyc-embed}
        Let \(G\) be a [f.g. free]-by-cyclic group with
        non-trivial centre and virtually cyclic abelianisation. 
        
        Then \(G\)
        embeds as a finite index subgroup in \(V \times \Z\) for some virtually
     free group \(V\). 
        Moreover, \(V \times \Z\) and the
        embedding of \(G\) into \(V \times \Z\) are computable from the
        presentation for \(G\).
    \end{proposition}

To prove this proposition, we use the fact that if a GBS group has non-trivial centre and virtually cyclic abelianisation, then the defining graph of the GBS group is a tree.
This follows from \cite{Levitt07}*{Theorem 1.1 (1)}.
    \begin{proof}
    The group $G$ is a GBS group, as seen in Section \ref{seq:equivalences}. Moreover, by 
         \cite{Levitt07}{Theorem 1.1 (1)} it is a GBS based on a tree, so there is a rooted tree \(T\) with root \(0\)
        such that \(G\) admits the presentation
        \[
            \langle \{a_v\mid v \in V(T)\} \mid \{a_{s(e)}^{p_e} = a_{t(e)}^{q(e)}
            \mid e \in E(T)\} \rangle,
        \]
        for some \(p_e, q_e \in \Z_{> 0}\) for all \(e \in E(T)\). In addition,
        \cref{lem:helpful2} tells us that this presentation is computable from a
        presentation for \(G\).

        Note that for all \(v \in V(T)\), \(a_0^{\prod_{e \in E(T)} p_{e}} =_G
        a_v^{n_v}\), for some \(n_v \in \Z\). Thus \(a_0^{\prod_{e \in E(T)}
        p_{e}} \in Z(G)\), and so  \(\langle a_0^{\prod_{e \in E(T)} p_{e}}
        \rangle \trianglelefteq G\). Set \(V = G / \langle a_0^{\prod_{e \in
        E(T)} p_e} \rangle\), and note that a presentation for \(V\) is
        computable from the GBS presentation for
        \(G\). We will show that \(V\) is virtually free, by showing that it is
        the fundamental group of a finite graph of groups, with finite vertex
        and edge groups. Such groups are always virtually free.

        Since for all \(v \in V(T)\), we have \(a_0^{\prod_{e \in E(T)}
        p_{e}} =_G a_v^{n_v}\)
        for some \(n_v \in \Z\), so the subgroup \(\langle
        a_v \rangle\) has order at most \(n_v\) in \(V\), and thus we can express
        \(V\) using the following presentation, where we let
        \(n_0 := \prod_{e \in E(T)} p_e\).
        We have replaced \(a_v\) with \(x_v\) to avoid confusion between
        \(V\) and \(G\).
        \[
            \langle \{x_v \mid v \in V(T)\} \mid \{x_{s(e)}^{p_e} = x_{t(e)}^{q_e} 
            \mid e \in E\} \cup \{x_v^{n_v} \mid v \in V(T)\} \rangle.
        \]
        This presentation is just that of the fundamental group of a graph of
        groups with underlying graph \(T\), and where the set of vertex groups
        is \(\{\langle x_v \mid x_v^{n_v} = 1\rangle \mid v \in V(T)\}\). This is a
        finite collection of finite groups, and so the fundamental group,
        which is \(V\), must therefore be virtually free.

        We next show that \(G\) embeds with finite index in \(V \times \langle d
        \mid \rangle\cong V \times \Z\), and show that the embedding is computable. To do so, we
        must first define \(Q = \langle t \mid t^{\prod_{e \in E(T)} p_e q_e}
        \rangle\) and maps \(\sigma \colon V \to Q\) and \(\tau \colon \langle d
        \rangle \to Q\) by \(\sigma(x_v) = t^{\prod_{e \in E(T), e < v} p_e
        \cdot \prod_{e \in E(T) e \not < v} q_e}   \) and \(\tau(d) = t\). Next
        define \(S \leq V \times \langle d \mid \rangle\) by \(S = \{(v, d^p)
        \in V \times \langle d \mid \rangle \mid \sigma(v) = \tau(d^p)\}\).
        Finally, define \(\psi \colon G \to V \times \langle d \mid \rangle\) by
        \(\psi(a_v) = (x_v, d^{\prod_{e \in E(T), e < v} p_e \cdot \prod_{e \in
        E(T) e \not < v} q_e})\), and note that \(\psi\) is computable from the
        generalised Baumslag-Solitar presentation for \(G\). We shall show that
        \(\psi\) is a momomorphism with \(\im \psi = S\), leaving it remaining
        to show that \(S\) has finite index in \(V \times \langle d \mid
        \rangle\).

        We begin by showing that \(\psi\) is a homomomorphism, which we
        do by showing that it satisfies the relations of \(G\). Let
        \(f \in E(T)\). Then 
        \begin{align*}
            (\psi(a_{s(f)}))^{p_f}
            & =_V  (x_{s(f)}^{p_f}, d^{p_f \cdot \prod_{e \in E(T), e < s(f)} p_e \cdot \prod_{e \in
        E(T) e \not < s(f)} q_e}) \\
            & =_V (x_{t(f)}^{q_f},  d^{q_f \cdot \prod_{e \in E(T), e < t(f)} p_e \cdot \prod_{e \in
        E(T) e \not < t(f)} q_e}) \\
        & = (\psi(a_{t(e)}))^{q_e}.
        \end{align*}
        Thus \(\psi\) is a homomorphism.

        We next show that \(\im \psi \subseteq S\). Since \(S\) is a subgroup,
        it suffices to show that \(\psi(a_v)\in S\) for all \(v \in V(T)\). So
        let \(v \in V(T)\). Then \(\psi(a_v) = (x_v, d^{\prod_{e \in E(T), e <
        v} p_e \cdot \prod_{e \in E(T) e \not < v} q_e})\). In addition,
        \(\sigma(x_v) = t^{\prod_{e \in E(T), e < v} p_e \cdot \prod_{e \in E(T)
        e \not < v} q_e} = \tau(d^{\prod_{e \in E(T), e < v} p_e \cdot \prod_{e
        \in E(T) e \not < v} q_e})\). Thus \(\psi(a_v) \in S\), and \(\im \psi
        \subseteq S\).

        To show \(\im \psi = S\), it thus suffices to show that \(S \subseteq
        \im \psi\). So let \((h, d^p) \in S\). Let \(f \colon G \to V\) be the
        natural quotient map, and note that \(\psi\) maps an element \(g \in G\)
        to \((f(g), d^q)\), for some \(q \in \Z\). Since \(f\) is surjective, we
        have that there exists \(g \in G\) such that \(\psi(g) = (g, d^q)\), for
        some \(q \in \Z\). We next show that \(p \equiv q \mod \prod_{e \in E(T)} p_e q_e\). Since \((h, d^p), (h, d^q) \in S\), we have
        that \(\sigma(h) =_Q \tau(d^p) =_Q \tau(d^q)\). Thus \(t^p =_Q t^q\),
        and so \(p \equiv q \mod \prod_{e \in E(T)} p_e q_e\). It
        remains to find an element \(k \in G\) such that \(\psi(k) = (1, d^{p -
        q})\). Since \(p - q \equiv 0 \mod \prod_{e \in E(T)} p_e q_e\), we have that
        \(x_0^{\frac{p - q}{\prod_{e \in E(T)} q_e}} =_V 1\), and so 
        \begin{align*}
           \psi(a_0^\frac{p - q}{\prod_{e \in E(T)} q_e}) 
           & =_{V \times \langle d \mid \rangle} 
        (x_0^\frac{p - q}{\prod_{e \in E(T)}{q_e}},  d^\frac{(p - q) \cdot \prod_{e \in E(T), e < 0} p_e \cdot
        \prod_{e \in E(T), e \not < 0} q_e}{\prod_{e \in E(T)} q_e}) \\
           & =_{V \times \langle d \mid \rangle} 
        (1, d^\frac{(p - q) \cdot 
        \prod_{e \in E(T)} q_e}{\prod_{e \in E(T)} q_e}) \\
        & =_{V \times \langle d \mid \rangle} (1, d^{p - q}) 
        \end{align*} 
        Thus \(\psi(ga_0^\frac{p - q}{\prod_{e \in E(T)} q_e}) =_{V \times \langle
        d \mid \rangle} (h, d^q)(1, d^{p - q}) =_{V \times \langle d \mid \rangle}
        (h, d^p)\). We can conclude that \(\im \psi = S\).

        To complete the proofs of our claims about \(\psi\), it remains to show
        that \(\psi\) is injective. Let \(g \in \ker \psi\). Since \(\psi\) maps
        \(g\) to \(1\) in the \(V\)-coordinate, we have that \(g \in \ker f =
        \langle a_0^{n_0} wrangle\), and so \(g = a_0^{p n_0}\) for some \(p \in
        \Z\). Thus \(\psi(g) = (1, d^{pn_0 \prod_{e \in E(T)} q_e})\). As \(n_0 \neq 0\) and \(q_e \neq 0\) for all \(e \in E(T)\), it thus follows
        that equating the second coordinate in the expression for \(\psi(g)\) to
        \(1\) implies that \(p = 0\). Thus \(g = 1\) and \(\psi\) is injective.
%        
%        For \(v \in V(T)\), write \(\ell_v = \prod_{e \in E(T), e < v} p_e \cdot
%        \prod_{e \in E(T), e \not < v} q_e\) (and so
%        \(\psi(a_v) = (x_v, d^{\ell_v})\)). Let \(u\) be any word representing
%        \(g\) over the generators of \(G\) and their inverses. From the fact that \(\psi(g)\) is trivial in the second coordinate
%        second coordinate, we have that by applying \(\psi\) to \(u\) letterwise, we must have that
%        \(\sum_{v \in V(T)} \ell_v (\#{a_v}(u) - \#_{a_v^{-1}}(u)) = 0\).
%        Plugging in \(g = a_0^{p n_0}\) into this
%        expression gives \(\ell_0 p n_0 = 0\).  Since \(\ell_0, n_0 > 0\), we have that
%        \(p =  0\). Thus \(g = 1\), and \(\psi\) is injective.

        Since \(\psi \colon G \to V \times \langle d \mid \rangle\), with \(\im
        \psi = S\), we have that \(G \cong S\). Thus to complete the proof, it
        remains to show that \([V \times \langle d \mid \rangle :S] < \infty\).
        Define \(\phi \colon V \times\langle d \mid \rangle \to Q \times Q\) by
        \(\phi(h, d^p) ={Q \times Q} (\sigma(h), \tau(d^p))\). Let \((h, d^p)
        \in \ker \phi\). Then \(\sigma(h) =_Q 1 =_Q \tau(d^p)\), and so \((h,
        d^p) \in S\). Thus \(\ker \phi \leq S\), and so \([G :S] \leq [G: \ker
        \phi] < \infty\).
    \end{proof}

    It is not difficult to show that the commutator subgroups of our groups are
    indeed finitely generated, and thus rational. We include the short proof below
    for completeness.
    \begin{lem}
        \label{fbc-centre-comm-sbg}
        Let \(G\) be a [f.g. free]-by cyclic group with centre and virtually
        cyclic abelianisation. Then the commutator subgroup \([G, G]\) is finitely generated free.
        Moreover, a finite generating set for \([G, G]\) can be computed from
        a presentation for \(G\).
    \end{lem}

    \begin{proof}
%        By \cite{Pietrowski}*{Theorem 1}, \(G\) is a generalised
%        Baumslag-Solitar group with centre, and thus by
%        \cite{Levitt15}*{Proposition 4.1}, \(G \cong F \rtimes_\phi \Z\) for
%        some finitely generated free group \(F\) and some automorphism \(\phi
%        \in \Aut(F)\) that has finite order in \(\Out(F)\). Moreover, using
%        \cref{lem:helpful2}, we can compute \(F\) and \(\phi\) from a
%        presentation for \(G\).
%
        Since \(G\) has a virtually cyclic abelianisation, \(G / [G, G] \cong \Z
        \times A\), for some finite abelian group \(A\). Note that \(A\) is
        computable from the presentation. Noting that \([G, G] \subseteq F\)
        and so \([G, G] \cap \Z = \{1\}\), we have \(G / [G, G] \cong F
        \rtimes_\phi \Z / [F \rtimes_\phi \Z, F \rtimes_\phi \Z] \cong F / [F
        \rtimes_\phi \Z, F \rtimes_\phi \Z] \times \Z\). Therefore, \(F / [G, G]
        \cong A\). Thus \([G, G]\) is a finite-index subgroup of \(F\), and so
        is finitely-generated free. Moreover, by enumerating all
        finite-index subgroups \(H\) of \(F\), checking if they are normal,
        and then when they are, computing the quotient \(F / H\), and checking
        if it is isomorphic to \(A\), we can compute a finite generating set for
        \([G, G]\).
    \end{proof}

    The second requirement of \cref{dir-prod-rat-constraints} is that
    the rational constraint (in our case the commutator subgroup) be
    fully contained in one of the factors of the direct product. The proof
    of the following firstly makes that observation and then applies
    \cref{dir-prod-rat-constraints}.

    \begin{proposition}
        \label{embed-into-VxA-plus-fg-comm-decidable-DP-ab}
        Let \(G\) be a finitely generated group such that \([G, G]\) is finitely
        generated and such that \(G\) embeds as a finite index subgroup into \(V
        \times A\) for some virtually (f.g.) free group \(V\) and
        (f. g.) abelian group \(A\). 
        
        If \(V \times A\),
        the embedding of \(G\) into \(V \times A\) and the finite generating
        set for \([G, G]\) are considered as input or computable,
        then \(G\) has decidable \(\DPAB{G}\). 
    \end{proposition}
    
    \begin{proof}
        First note that \([G, G] \subseteq [V \times A, V \times A] = [V, V]
        \subseteq V\). Since \([G, G]\) is finitely generated, it is a rational
        subset of \(V \times A\) that is fully contained in \(V\). Next, notice
        that as \(G\) has finite index in \(V \times A\), it is a recognisable set.

        Let \(\mathcal E_0\) be a system of equations with abelianisation
        constraints in \(G\). Since \(G \leq V \times A\) with finite index,
        this can be translated into a system of equations \(\mathcal E_1\) in
        \(V \times A\) with the recognisable constraints that every variable
        lies in \(G\), along with the rational constraints that \(X \in [G, G]\)
        for each abelianisation constraint in \(\mathcal E_0\).

        By \cite{DiekertElder}*{Theorem 14.2}, \(\DPRat{V}\) and
        \(\DPRat{A}\) are both decidable and EDT0L.
        Thus \cref{dir-prod-rat-constraints} implies that
        the Diophantine Problem in \(V \times A\) with rational constraints
        fully contained in \(V\) is EDT0L. To add the recognisable constraints,
        apply \cite{Levine23}*{Proposition 6.1}.
    \end{proof}

    We can now combine the results from this section to obtain the following.

    \begin{theorem_no_number}[\ref{fbyc-centre-decidable}]
       Let \(G\) be a [finitely-generated free]-by-cyclic group with centre and virtually cyclic
       abelianisation. Then \(\DPAB{G}\) is decidable and EDT0L.
    \end{theorem_no_number}

    \begin{proof}
        By 
        \cref{embed-into-VxA-plus-fg-comm-decidable-DP-ab}, it suffices to show
        that \([G, G]\) is finitely generated with a finite generating set
        computable from \(G\) and \(G\) embeds with finite index into \(V \times
        \Z\), for some virtually finitely generated free group \(V\), with \(V
        \times \Z\) and the embedding computable from the presentation for
        \(G\). The former follows by \cref{fbc-centre-comm-sbg} and the
        latter by \cref{fbyc-embed}.
    \end{proof}

    Since the class of one-relator groups with centre is a proper subclass of
    the class of [finitely-generated free]-by-cyclic groups with centre, we
    have the following.
    
    \begin{cor_no_number}[\ref{orgwc-decidable-cor}]
       Let \(G\) be a one-relator group with centre and virtually cyclic
       abelianisation. Then \(\DPAB{G}\) is decidable and EDT0L. 
    \end{cor_no_number}

    \begin{proof}
        By \cite{Pietrowski}*{Theorem 1}, one-relator groups with centre and virtually
        cyclic abelianisation are GBS groups with centre. We can moreover use
        an identical argument to that provided in \cref{lem:helpful2}
        to compute a free-by-cyclic presentation for a one-relator group with
        centre from a one-relator presentation. The result thus follows
        from \cref{fbyc-centre-decidable}.
    \end{proof}

\section{Tree products of infinite cyclic groups and tree-GBS groups} \label{sec:treeprod-GBS}
In the previous section we made use of the fact that any [f.g. free]-by-cyclic group with non trivial centre and virtually cyclic abelianisation is isomorphic to a GBS-group where, as a graph of groups, the underlying graph is a tree. 
So a tree-GBS group is a graph of groups where all vertex and edge groups are $\Z$, and the underlying graph of the graph of groups is a tree. Note that graphs of groups where the graph is a tree are usually called tree of groups in the literature. 
So a tree-GBS group is a tree of groups with infinite cyclic edge and vertex groups.  
Members of this class are called \emph{tree-GBS groups} in the literature. The tree-GBS groups are also closely related to another class of groups, namely tree products of infinite cyclic groups. 
A tree product of infinite cyclic groups is a graph of groups where all the vertex groups are $\mathbb{Z}$ and the underlying graph is a tree. Thus the only difference with tree-GBS groups is that the edge groups can be either $\Z$ or $\{ 1 \}$ in a tree product of infinite cyclic groups. 
It is immediate from the definitions that any tree product of infinite cyclic groups is a free product of tree-GBS groups, where the free product decomposition is obtained by deleting all the edges labelled by $\{1\}$ in the graph product description of the group.   
In this section we explore the connections with tree-GBS groups and more generally for tree products of infinite cyclic groups. 
Tree products of infinite cyclic groups have been studied e.g. in \cite{Picantin} where is proved that every tree product of infinite cyclic groups is a Garside group. 
Some papers on tree-GBS groups include e.g. \cite{Dudkin24} and 
\cite{JonesLogan25}.  

First we observe that the class of free-by-cyclic groups considered in the
previous section are exactly the tree-GBS groups. 
\begin{proposition}\label{prop:TreeGBP} 
The following are equivalent for a finitely generated group $G$:
\begin{enumerate} 
\item $G$ is a tree-GBS group; 
\item $G$ is a [f.g. free]-by-cyclic group with non trivial centre and virtually cyclic abelianisation.
\end{enumerate}  
\end{proposition}
\begin{proof} 
It follows from Lemma~\ref{lem:helpful2} that if 
$G$ is a [f.g. free]-by-cyclic group with non trivial centre and virtually cyclic abelianisation then $G$ is a tree-GBS group. 

For the converse, let $G$ be a tree-GBS group.
The fact that every tree-GBS group has non-trivial center is proved in the first two paragraphs of the proof of Proposition~\ref{fbyc-embed}. 
Alternatively the fact that every tree-GBS group has non-trivial center can be deduced from 
\cite{Levitt07}*{Theorem 1.1}
%[Theorem 1.1, Levitt, G. (2007). On the automorphism group of generalized Baumslag–Solitar groups. Geometry \& Topology, 11(1), 473-515.] 
together with the fact that the first Betti number of a tree is $0$.  
We already say above that it is a consequence of \cite{Lev15} that the class of GBS groups with non-trivial center is equal to the class of   
[f.g. free]-by-cyclic group with non trivial centre. 
Hence it follows that $G$ is a [f.g. free]-by-cyclic group with non trivial centre.  
To complete the proof we just need to verify that every tree-GBS group has
a virtually cyclic abelianisation. This follows from \cite{Levitt07}*{Theorem 1.1}
%, Levitt, G. (2007). On the automorphism group of generalized Baumslag–Solitar groups. Geometry \& Topology, 11(1), 473-515. ]
or alternatively from \cite{Robinson11}*{Theorem 2} 
%, Robinson, D. J. S. (2011). The Schur Multiplier of a Generalized Baumslag-Solitar Group. Rendiconti del Seminario Matematico della Universita di Padova]
which are general results that compute the torsion-free rank of the abelianisation of a GBS group.    
\end{proof}

In general, the decidability of the Diophantine Problem is not preserved when taking graphs of groups. Some positive results exist in this direction: see, for example \cite{LohreySenizergues}. On the other hand there are even GBS groups with undecidable Diophantine Problems \cite{BobAlexinprep}. 
The following result show that for tree products of infinite cyclic groups the Diophantine Problem is decidable, and for this class the result classifies exactly when \(\DPAB{G}\) is decidable. 

\begin{theorem_no_number}[\ref{thm:TreeProduct}] 
Let $G$ be a tree product of infinite cyclic groups.  
Then $G$ has decidable Diophantine Problem. %%jkkwith EDT0L solutions. 
Furthermore:  
\begin{enumerate} 
\item if all edge groups are infinite cyclic then \(\DPAB{G}\) is decidable and EDT0L; 
\label{tree-prod-dec}
\item otherwise, \(\DPAB{G}\) is undecidable. 
\label{tree-prod-undec}
  \end{enumerate}
\end{theorem_no_number}
\begin{proof} 
Since $G$ is a tree product of infinite cyclic groups it follows that $G$ is a
free product of tree-GBS groups. From \cref{fbyc-embed} and \cref{prop:TreeGBP}, each of the individual
tree-GBS groups is virtually a direct product of hyperbolic groups, and thus
has a decidable existential theory by \cite{equations_VDP}*{Theorem 3.1}. We can
thus apply \cite{LohreySenizergues}*{Theorem 1} which shows stability of the existential
theory under free products, to conclude that \(G\) has a decidable existential theory,
and hence Diophantine Problem.

For \eqref{tree-prod-dec}, if all the edge groups are infinite cyclic then $G$
is a tree-GBS group and thus by Proposition~\ref{prop:TreeGBP} and
Theorem~\ref{fbyc-centre-decidable} it follows that \(\DPAB{G}\) is decidable
and EDT0L.

It remains to prove \eqref{tree-prod-undec}. For this we need to show if at
least one of the edge groups is trivial then \(\DPAB{G}\) is undecidable. Let
\(T\) be the underlying tree in the graph of groups defining \(G\), and let
\(\bar{T}\) be the forest obtained from \(T\) be deleting all edges whose edge
groups are trivial. Let \(\bar{T}_1, \ldots, \bar{T}_m\) be the connected
components of \(\bar{T}\). From our case assumption, at least one such edge must
exist, and so \(\bar{T}\) has at least two connected components; that is, \(m
\geq 2\). We have that \(G \cong \pi_1(\bar{T}_1) \ast \cdots \ast
\pi_1(\bar{T}_m)\). 
Moreover,
by definition, each \(\pi_1(\bar{T}_i)\) is a tree-GBS group, and thus has a
virtually infinite cyclic abelianisation \cite{Levitt07}*{Theorem 1.1}. Since we
additionally have that \(\ab(G) = \ab(\pi_1(\bar{T}_1)) \times \cdots \times
\ab(\pi_1(\bar{T}_m))\), it follows that for each \(i\) there exists \(v_i \in
\pi_1(\bar{T}_i)\) such that \(\langle \ab_G(v_i) \rangle \cong \Z\).

Now consider the subgroup \(H = \langle v_1, \ldots, v_m \rangle\) of \(G\).
Since each \(v_i\) projects onto an infinite order element in \(\ab(G)\), we
must have that each \(v_i\) has infinite order in \(G\). As each lie in a
distinct factor in the free product defining \(G\), it follows that \(H\) is
free of rank \(m\). To complete our proof, we show that any algorithm that
can decide \(\DPAB{G}\) can decide \(\DPAB{H}\). Since the latter is
undecidable \cite{CiobanuGarreta}*{Theorem A}, this will complete the proof.
To do so, we first construct a retract from \(G\) onto \(H\). We do this
by constructing retracts \(r_i \colon \pi_1(\bar{T}_i) \to \langle v_i \rangle\)
for each \(i\). Let \(i \in \{1, \ldots, m\}\). Then \(\ab_G(\pi_1(\bar{T}_i))\)
is virtually \(\Z\) \cite{Levitt07}*{Theorem 1.1}, and so retracts onto the copy of \(\Z\) inside it.
Composing this retract with \(\ab_G\) gives a function from \(\pi_1(\bar{T}_i)\) to
\(\Z\) that maps \(v_i\) to the generator. It thus follows that there is a retract
\(r_i \colon \pi_1(\bar{T}_i) \to \langle v_i \rangle\). As \(G\) is a free
product of the groups \(\pi_1(\bar{T}_i)\), these now define a unique retract \(r
\colon G \to H\) that extends each of the retracts \(r_i\).

Now suppose that \(\mathcal{E}\) is a system of equations in \(H\) with
abelianisation constraints. We can assume every equation
in \(\mathcal E\) is of the form \(XY = Z\), where \(X\), \(Y\) and \(Z\) can be
either constants or variables (see for example \cite{CiobanuDiekertElder}*{Section 4}). We will use \(\mathcal E\) to define a new system
\(\bar{\mathcal E}\) of equations with abelianisation constraints in \(G\) that
admits a a solution if and only if \(\mathcal E\) does.  For each equation \(XY =
Z\) in \(H\), add the same equation to \(\bar{\mathcal E}\). To consider the abelianisation constraints, write \(\ab(G) = \Z^m \times A\),
where \(A\) is a finite abelian group, and let \(p\) be the exponent of \(A\). For
each
abelianisation constraint \(X \in [H, H]\), add the abelianisation constraint
\(X^p \in [G, G]\) to \(\bar{\mathcal{E}}\). Naturally, every solution to
\(\mathcal{E}\) will also be a solution to \(\bar{\mathcal E}\). Moreover, if
there exist \(x, y, z \in G\) such that \(xy = z\), then \(r(x) r(y) = r(z)\).
 If \(x \in G\) is such that \(x^p \in [G, G]\), then \(r(x)\) will lie in
 \([H, H]\) if and only if the exponent-sum of each \(v_i\) in \(r(x)\) is
 \(0\), which is true if and only if the exponent-sum of each \(v_i\) in
 \(r(x^p)\) is \(0\). By the definition of \(r\), this holds if and only if for each \(i\)
 the exponent-sum across all of the elements mapped by \(r\) to \(v_i\) is
 \(0\) (noting that the \(p\)th power ensures nothing in \(A\) can appear). But
 since \(x^p \in [G, G]\), this must hold, and so we can conclude that
 \(r(x) \in [H, H]\). We have thus shown that if \(\bar{\mathcal E}\) admits a
 solution, then applying \(r\) to this solution yields a solution to
 \(\mathcal E\). We can conclude that \(\mathcal E\) admits a solution if and only
 if \(\bar{\mathcal E}\) does. Since \(\bar{\mathcal E}\) is computable from
 \(\mathcal E\), we can thus use \(\DPAB{G}\) to solve \(\DPAB{H}\). Since
 the latter is undecidable \cite{CiobanuGarreta}*{Theorem A}, the result follows.
\end{proof}
 In  \cite{Levitt07}*{Theorem 1.1}
%Levitt, G. (2007). On the automorphism group of generalized Baumslag–Solitar
%groups. Geometry \& Topology, 11(1), 473-515.] 
and  \cite{Robinson11}*{Theorem 2},
%Robinson, D. J. S. (2011). The Schur Multiplier of a Generalized Baumslag-Solitar Group.
%Rendiconti del Seminario Matematico della Universita di Padova]
formulas are
given for the torsion free rank of the abelianisations of GBS groups, given in
terms of Betti numbers of their defining graphs. In particular those results may
be applied to describe exactly those GBS groups with virtually cyclic
abelianisation:   

Combining our results with Levitt's theorem describing the rank of the abelianisation
of a GBS group in terms of the graph of groups
\cite{Levitt07}*{Theorem 1.1} gives the following:  
%{\color{red} AL: Just to confirm - graphs in graphs of groups need to be connected, right (I think a spanning tree is needed to define the presentation)? 
%}

\begin{cor}
\label{prop:GBSWithVirtCyclicAbel} 
Let $G$ be a GBS group 
with virtually cyclic abelianisation. 
%%
%%such that the abelianisation of $G$ has torsion-free rank 1.
%%
Then one of the following holds: 
\begin{enumerate}
\item $G$ is a tree-GBS group and  \(\DPAB{G}\) is decidable and EDT0L; or else     
\item $G \cong BS(m,n)$ for some $m,n$ with $|m| \neq |n|$.   
\end{enumerate}
\end{cor}

The question for $BS(m,n)$ (case (2)) remains open. Indeed, it is not even know
whether the Diophantine Problem is decidable for any $BS(m,n)$ for some $m,n$
with $|m| \neq |n|$. It was however recently shown \cite{BobAlexinprep} that
every group $BS(m,n)$  with $1 < m < n$ is quasi-isometric to a group with an
undecidable Diophantine Problem. The addition of abelianisation constraints does not make a difference in the
case of \(BS(1, n)\), as here every element of the commutator subgroup is a product of
a bounded number of commutators, and so the commutator subgroups is definable in
the Diophantine Problem.

\section{Undecidability results for free-by-cyclic groups}\label{sec:undecidability}

In this section we exhibit free-by-cyclic groups $G$ for which \(\DPAB{G}\) is undecidable, providing a counterpoint to the classes of groups in Sections \ref{sec:fbyc-with-centre} and \ref{sec:treeprod-GBS}, where the problem is decidable.

The undecidability is connected both to the abelianisation being `larger' than virtually cyclic, and to the centre being trivial. 

\subsection{Free-by-cyclic groups with abelianisations of torsion-free rank at least 3} 

We show that in free-by-cyclic groups with centre, when the rank
of the torsion-free part of the abelianisation is greater than or equal to $3$, the
problems we have considered before become undecidable. A key tool to establish
this, and of independent interest, is Proposition \ref{qt-by-centre}.

%\lc{Braid groups have infinite cyclic centre, so try to apply the proposition below to them to get undecidability results for $\DPAB{G}$. In fact, this holds for all spherical Artin groups and we understand quite a bit about the quotient.}

%\lc{Maybe write down examples for which the converse to Proposition 5.1 is not true.}
%{\color{purple} AL: this is done after the proof for some reason. Happy to move it}
\begin{proposition}
    \label{qt-by-centre}
    Let \(G\) be a finitely generated group. 
    
    If \(\DPAB{G}\) is decidable, then
    \(\DPAB{G / Z(G)}\) is decidable.
\end{proposition}

\begin{proof}
    Let \(Q = G / Z(G)\). Fix a lift \(s \colon Q \to G\) that satisfies $s(1_Q)=1_G$. Then every element in
    \(G\) can be expressed in the form \(s(q) z\) for some \(q \in Q\) and \(z
    \in Z(G)\). Moreover, if \(s(q)z = s(q')z'\), then \(s(q) (s(q'))^{-1} \in
    Z(G)\), and so \(s(q) = s(q') g\) for some \(g \in Z(G)\). It follows that \(q = q'\) and \(z = z'\),
    and so this representation is unique. 
    
    Take \(X Y Z  =1\) to be an equation in \(Q\), where \(X\), \(Y\) and \(Z\) are variables or
    constants. Let \(\bar{X}\) just denote a variable in \(G\) if \(X\) is a
    variable, and denote \(s(X)\) if \(X\) is a constant. Define \(\bar{Y}\) and
    \(\bar{Z}\) analogously. Then \(XY Z = 1\) has a solution in \(Q\) if and
    only if the system with constraints \((\bar{X} \bar{Y} \bar{Z}) = U
    \wedge (U\in Z(G))\) has a solution. Moreover, the solutions to this new
    system are exactly the solutions to \(XYZ=1\) in \(Q\) lifted by \(s\) and
    then multiplied by some element of \(Z(G)\). Note that \(Z(G)\) is
    \(\DP{G}\)-definable (see Example \ref{ex:central}(2)). Thus this system with constraints \((\bar{X} \bar{Y}
    \bar{Z} = U) \wedge (U\in Z(G))\) is equivalent to a (unconstrained) system
    in \(G\). 

    At this point, we have shown that the decidability of \(\DPAB{G}\) implies
    the decidability of \(\DP{Q}\), so it remains to consider the
    abelianisation constraints in \(Q\). Let \(q \in Q\). We claim that
    \(q \in [Q, Q]\) if and only if \(s(q) \in [G, G]Z(G)\). First, by applying the
    projection homomorphism from \(G\) to \(Q\), \(s(q) \in [G, G] Z(G)\) implies
    \(q \in [Q, Q]\). Second, suppose that \(q \in [Q, Q]\). Then \(q = c_1 \cdots c_m\),
    for commutators \(c_1, \ldots, c_m \in Q\). Thus \(s(q) = s(c_1) \cdots s(c_m) z\),
    for some \(z \in Z(G)\), and so \(s(q) \in [G, G] Z(G)\), as required. As a
    product of two \(\DPAB{G}\)-definable sets, \([G, G] Z(G)\) is also
    \(\DPAB{G}\)-definable. We can thus replace any abelianisation constraint
    \(X \in [Q, Q]\), where \(X\) is a variable with \(\bar{X} \in [G, G] Z(G)\) to
    complete our encoding of \(\DPAB{Q}\) into \(\DPAB{G}\). 
\end{proof}

\begin{rmk}
     Note that the reverse direction in Proposition~\ref{qt-by-centre} does not hold: the $3 \times 3$ Heisenberg group over the integers is a counterexample as a central extension
of $\Z^2$ \cite{DuchinLiangShapiro}. In fact, there is a finitely generated class \(2\) nilpotent
group (and thus a central extension of a finitely generated abelian group) with undecidable
single equation problem \cite{RomankovSEP}, and thus even in case of single equations the converse to the above fails to hold.
\end{rmk}
\begin{theorem_no_number}[\ref{fbyc-big-ab}]
    Let \(G\) be a [f.g. free]-by-cyclic group with non-trivial centre and a torsion-free
    rank \(3\) or greater abelianisation. Then \(\DPAB{G}\) is undecidable.
\end{theorem_no_number}

\begin{proof}
    %Write \(G = F \rtimes \langle t \rangle\), where \(\langle t \rangle \cong\Z\). 
    By
    \cite{BridsonPiwek}*{Corollary 3.7} the group \(G / Z(G)\) is
    virtually free. As the rank of the abelianisation of \(G\) is \(3\) or more,
    and \(Z(G)\) is cyclic (\cite{BridsonPiwek}*{Lemma 3.4})
    the rank of the abelianisation of \(G / Z(G)\) is \(2\) or more. We can
    therefore apply \cite{CiobanuGarreta}*{Theorem B} to conclude that \(\DPAB{G
    / Z(G)}\) is undecidable, and so \cref{qt-by-centre} tells us that
    \(\DPAB{G}\) is undecidable.
\end{proof}

We now consider various cases without centre, starting with the first hydra group.

\subsection{The first hydra group}
    The object of study in this section is a free-by-cyclic group with
    abelianisation equal to $\Z^2$ and trivial centre, more commonly given by
    its one-relator presentation \[K = \langle c, q \mid c(qcq^{-1}) =
    (qcq^{-1})c \rangle.\] We will show that \(\DPAB{K}\) is undecidable by
    using the isomorphism between $K$ and $G$ given in Lemma \ref{lem:H-K-iso}
    and working with a free-by-cyclic presentation $F \rtimes_{\phi} \Z$ for
    this group, where $\phi$ sends $a$ to $ab$ and $b$ to $b$. Note that
    $\phi$ has infinite order in $Out(F)$, so by \cite{BridsonPiwek}*{Lemma
    3.4} the centre of $K$ is trivial. We note that it is unknown
    whether $\DP{K}$ is decidable. 
    
    This is the
    `first' hydra group defined by Dison and Riley \cite{HydraGroups}, and
    enjoys various nice properties such as \(\operatorname{CAT}(0)\) and
    biautomatic. We show that this group has an undecidable Diophantine Problem
    with abelianisation constraints. We expect the fundamental idea of
    our proof will work for all hydra groups, only our method for computing
    centralisers does not seem to, so some alternative approach for computing
    these centralisers will be needed.

    We will mostly be using the free-by-cyclic description of \(K\) given below as $G$.
    However, in order to compute certain centralisers, it is easier to view
    \(K\) as the semidirect product of the right-angled Artin group on a
    bi-infinite line with \(\Z\), so we give the isomorphism between the two
    standard presentations for \(K\).
  \begin{lem}[{\cite{BobAlexinprep}*{Lemma 5.4}}]
    \label{lem:H-K-iso}
    There is an isomorphism \(\Theta\) from $$G= \langle a, b, t \mid tat^{-1} = ab, tbt^{-1} = b\rangle$$ to \(K = \langle c, q \mid c(qcq^{-1}) = (qcq^{-1} c) \rangle\) defined by $a \mapsto cq^{-1}$, $b \mapsto qcq^{-1} c^{-1}$, and $t\mapsto c.$ 
    %\((a)\Theta = cq^{-1}, \quad (b)\Theta = qcq^{-1} c^{-1}, \quad (t)\Theta = c.\)
    
    In addition, let \(c_i = q^i c q^{-i}\), where \(i \in \mathbb{Z}\). Then
    \(A(\Gamma) = \langle \{c_i \mid i \in \mathbb{Z}\} \rangle\) is the
    right-angled Artin group on an infinite line, and \[K \cong A(\Gamma) \rtimes_\phi \Z,\] where
    \(\phi \in \Aut(A(\Gamma))\) is defined by \(c_i \mapsto c_{i + 1}\).
  \end{lem}

    Our proof of undecidability requires various sets to be \(\DPAB{G}\)-definable.
    Centralisers provide a useful source for these, since they are
    always \(\DPAB{G}\)-definable (they are, in fact \(\DP{G}\)-definable). We prove
    this using the normal form for semidirect products: either that for the
    \(F_2\)-by-\(\Z\) presentation or that for the \(A(\Gamma)\)-by-\(\Z\)
    presentation.
  \begin{lem} 
    Let \(G = \langle a, b, t \mid tat^{-1} = ab, tbt^{-1} = b\rangle\) and \(K
    = \langle c, q \mid c(qcq^{-1}) = (qcq^{-1})c \rangle\), and let \(c_i\) be as
    defined in \cref{lem:H-K-iso}. Then
    \label{lem:Bob-gp-cent}
    \begin{multicols}{2}
    \begin{enumerate}
       \item \(C_G(a) = \langle a \rangle\);
       \label{item:L-cent-a}
       \item \(C_G(b) = \langle b, t \rangle\);
       \label{item:L-cent-b}
       \item \(C_G(aba^{-1}) = \langle aba^{-1}, t \rangle\);
       \label{item:L-cent-abA}
       \item \(C_K(c_i) = \langle c_{i - 1}, c_i, c_{i + 1} \rangle\);
       \label{item:H-cent-ci}
       \item \(C_G(t) = \langle aba^{-1}, b, t \rangle\).
       \label{item:L-cent-r}
    \end{enumerate}
    \end{multicols}
  \end{lem}
  
  \begin{proof}
    \eqref{item:L-cent-a}: Let \(ut^k \in G\), where $u$ is a word over $\{a^{\pm 1}, b^{\pm 1}\}$. Then \(u t^k \in C_G(a)\)
    if and only if \(aut^k =_G ut^k a =_G uab^k t^k\), which holds if and only if
    \(au =_{F(a, b)} uab^k\). This can only happen if \(k = 0\), as seen by taking the abelianisation on both sides, and so \(au =_{F(a, b)} ua\), which implies \(u\) is a power of \(a\). It follows that
    \(C_G(a) = \langle a \rangle\).

    \eqref{item:L-cent-b}: First note that \(b, t \in C_G(b)\), and so
    \(\langle b, t \rangle \subseteq C_G(b)\). Now let \(ut^k \in C_G(b)\).
    Then \(but^k = ut^k b = ubt^k\), and so \(ub = bu\); that is, \(u \in
    C_{F(a, b)}(b) = \langle b \rangle\). Thus \(ut^k = b^l t^k \in \langle b, t
    \rangle\), and so \(C_G(b) = \langle b,t 
    \rangle\).

    \eqref{item:L-cent-abA}: We have that
    \(taba^{-1} = ab \cdot b \cdot b^{-1} a^{-1} t = aba^{-1}t\), and so
    \(\langle aba^{-1}, t \rangle \subseteq C_G(aba^{-1})\). Let \(ut^k \in
    C_G(aba^{-1})\). Then \(aba^{-1} \cdot ut^k =_G ut^k \cdot aba^{-1} =_G u aba^{-1} t^k\).
    Thus \(uaba^{-1} =_{F(a, b)} aba^{-1}u\), and so \(u \in C_{F(a, b)}(aba^{-1}) =
    \langle aba^{-1} \rangle\); then \(ut^k \in \langle
    aba^{-1}, t\rangle\), and \(C_G(aba^{-1}) = \langle aba^{-1},t 
    \rangle\).

    \eqref{item:H-cent-ci}: From the presentation of
    \(A(\Gamma)\), we have that \(\langle c_{i-1}, c_{i}, c_{i + 1} \rangle
    \subseteq C_K(c_i)\). Now let \(u q^k \in C_K(c_i)\), where \(u \in
    A(\Gamma)\) and \(k \in \mathbb{Z}\). Then \(c_i uq^k =_K uq^k c_i =_K uc_{k
    + i} q^k\), and so \(c_i u =_{A(\Gamma)} u c_{k + i}\); that is \(c_i\) is
    conjugate in \(A(\Gamma)\) to \(c_{k + i}\). It follows that \(c_i\) and
    \(c_{k + i}\) are conjugate in the abelianisation of \(A(\Gamma)\); that is,
    they are equal in the abelianisation. But generators in \(A(\Gamma)\) are
    mapped to generators in the infinite rank free group, and so the images of
    \(c_i\) and \(c_{i + k}\) cannot be equal in the abelianisation unless \(c_i
    = c_{i + k}\). Thus \(k = 0\), and so \(C_K(c_i) \subseteq A(\Gamma)\).
    Since the \(C_{A(\Gamma)}(c_i) = \langle c_{i - 1}, c_i, c_{i + 1} \rangle\)
    (\cite{Servatius89}*{The Centralizer Theorem}, see also
		\cite{DayWade19}*{Proposition 2.1}), the result follows.

        \eqref{item:L-cent-r}: We compute \(C_G(t)\) by considering its image in
    the (isomorphic) group \(K\), using the isomorphism from \cref{lem:H-K-iso}.
    So we instead need to compute \(C_K(c)\). From \eqref{item:H-cent-ci},
    \(C_K(c) = C_K(c_0) = \langle c_{-1}, c_0, c_1 \rangle = \langle
    q^{-1}cq, c, qcq^{-1} \rangle\). The preimage of this subgroup under the
    isomorphism from \cref{lem:H-K-iso} is \(\langle t^{-1} a t a^{-1} t,
    t, bt \rangle = \langle a b^{-1} a^{-1}t, t, bt \rangle =
    \langle aba^{-1}, b, t \rangle\).
\end{proof}

Using the centralisers we computed, we now have that the following sets
are definable using the Diophantine Problem in \(G\).

\begin{lem}
    \label{Bob-gp-defable}
    Let \(G = \langle a, b, t \mid tat^{-1} = ab, tbt^{-1} = b\rangle\). Then
    the following subsets of \(G\) are \(\DP{G}\)-definable:
    \begin{multicols}{2}
    \begin{enumerate}
      \item \(\langle a \rangle\);
      \label{item:a-definable}
      \item \(\langle b, t \rangle\);
      \label{item:b-r-definable}
      \item \(\langle aba^{-1}, t \rangle\);
      \label{item:abA-r-definable}
      \item \(\langle t \rangle\).
      \label{item:r-definable}
    \end{enumerate}
    \end{multicols}
  \end{lem}

  \begin{proof}
    As centralisers are \(\DP{G}\)-definable,
    \eqref{item:a-definable},
    \eqref{item:b-r-definable} and \eqref{item:abA-r-definable} analogously
    follow from \cref{lem:Bob-gp-cent}\eqref{item:L-cent-a}, \eqref{item:L-cent-b} and
    \eqref{item:L-cent-abA}.
    For \eqref{item:r-definable}, note that \(\langle t \rangle =
    \langle aba^{-1}, t \rangle \cap \langle b, t \rangle\). Since intersections
    of \(\DP{G}\)-definable sets are \(\DP{G}\)-definable 
    and both of these sets are \(\DP{G}\)-definable (\eqref{item:b-r-definable} and
    \eqref{item:abA-r-definable}), the result follows.
%
%    \eqref{item:b-definable}: Using \cref{lem:definable-closure-props}, it
%    suffices to show that \(a \langle b \rangle\) is equationally definable.
%    Since \(\langle r \rangle\) is definable by \eqref{item:r-definable}, let
%    \(\mathcal E_r\) be a system of equations with a variable \(X_r\) such that
%    the set of solutions to \(X_r\) is \(\langle r \rangle\). Then the set of
%    solutions to the variable \(Y\) within the system of equations \(\mathcal
%    E_r \cup \{Y = X_r a X_r^{-1}\}\) is \(\{ab^k \mid k \in \mathbb{Z}\} = a
%    \langle b \rangle\), as required.
\end{proof}
%
%\begin{lem}
%    \label{Bob-gp-defable}
%    The following sets are DP-definable in \(G = \langle a, b, t \mid tat^{-1}
%    = ab, tbt^{-1} = b \rangle\):
%    \begin{enumerate}
%        \item \(\langle a \rangle\);
%        \item \(\langle b \rangle\);
%        \item \(\langle t \rangle\).
%    \end{enumerate}
%\end{lem}
%
%\begin{proof}
%    {\color{red}
%    Done by Bob and Alex in some other document.} 
%\end{proof}

The next centraliser we use is not powerful enough on its own to give undecidability
in any obvious way. However, when we take its abelianisation, it does encode a divisibility
condition in the ring of integers, that will be important to reduce Hilbert's
10th problem in the ring of integers to \(\DP{G}\). We do not explicitly compute this
centraliser; just its abelianisation.

\begin{lem}
    \label{ab-cents-Bob-gp}
    Let \(G = \langle a, b, t \mid tat^{-1} = ab, tbt^{-1} = b \rangle\). Fix \(k\in \Z\). Then
     \[
        \ab(C_G(a t^k)) = \{a^\ell t^{k \ell} \mid \ell \in \Z\}.
    \]
\end{lem}

\begin{proof}
    We use the isomorphism and notation of \cref{lem:H-K-iso}. 
    % that \(G \cong K \coloneqq \langle c, q \mid c(qcq^{-1}) = (qcq^{-1})c \rangle\), via the isomorphism defined by \(a \mapsto cq^{-1}\), \(b \mapsto qcq^{-1} c^{-1}\) and \(t \mapsto c\). \cref{lem:H-K-iso} further tells us that \(K\) is isomorphic to \(A(\Gamma) \rtimes_\phi \Z\), where \(\Gamma\) is the bi-infinite line graph (that is, the Cayley graph of \(\Z\)), \(A(\Gamma)\) is the right-angled Artin group on \(\Gamma\) and \(\phi\) is the action of \(\Z\) on \(A(\Gamma)\) induced by shifting the vertices of \(\Gamma\) along by \(1\). We will use \(c_i = q^i c q^{-i}\) to denote the generators of \(A(\Gamma)\). 
    We have that
   \(C_G(at^k) = C_K(c_0 q^{-1} c_0^k) = C_K(c_{-1}^k c_0 q^{-1})\). Let \(u q^\ell \in C_K(c_{-1}^k
   c_0 q^{-1})\),
   where \(u \in A(\Gamma)\) and \(\ell \in \Z\). Then we use $u q^\ell \cdot c_{-1}^k c_0 q^{-1}= c_{-1}^k c_0 q^{-1} \cdot u q^\ell$ and rewrite as
   \begin{align*}
        u c_{\ell - 1}^k c_{\ell} q^{\ell - 1}
       & = u q^\ell \cdot c_{-1}^k c_0 q^{-1} \\
       & = c_{-1}^k c_0 q^{-1} \cdot u q^\ell \\
       & = c_{-1}^k c_0 \cdot (u) \phi^{-1} \cdot q^{\ell - 1}.
   \end{align*}
   Thus \(c_{-1}^k c_0 \cdot (u) \phi^{-1} = u c_{\ell - 1}^k c_{\ell} \) and so
   \((u)\phi^{-1} = c_0^{-1} c_{-1}^{-k} u c_{\ell - 1}^k c_{\ell}\). 

    Abelianising, noting that exponent-sums of the \(c_i\)s are well-defined inside the infinite
    right-angled Artin group, gives:
    \begin{enumerate}
        \item \(\expsum_{c_i}((u)\phi^{-1}) = \expsum_{c_i}(u)\), for all \(i \notin \{-1, 0, \ell - 1, \ell\}\);
        \item \(\expsum_{c_{-1}}((u)\phi^{-1}) = \expsum_{c_{-1}}(u) - k\);
        \item \(\expsum_{c_0}((u)\phi^{-1}) = \expsum_{c_0}(u) - 1\);
        \item \(\expsum_{c_{\ell - 1}}((u)\phi^{-1}) = \expsum_{c_{\ell - 1}}(u) + k\);
        \item \(\expsum_{c_\ell}((u)\phi^{-1}) = \expsum_{c_\ell}(u) + 1\).
    \end{enumerate}
    These each imply, respectively:
    \begin{enumerate}
        \item \(\expsum_{c_{i + 1}}(u) = \expsum_{c_i}(u)\), for all \(i \notin \{-1, 0, \ell - 1, \ell\}\);
        \item \(\expsum_{c_0}(u) = \expsum_{c_{-1}}(u) - k\);
        \item \(\expsum_{c_1}(u) = \expsum_{c_0}(u) - 1\);
        \item \(\expsum_{c_{\ell}}(u) = \expsum_{c_{\ell - 1}}(u) + k\);
        \item \(\expsum_{c_{\ell + 1}}(u) = \expsum_{c_\ell}(u) + 1\).
    \end{enumerate}
    From (1), we have \(\expsum_{c_{\ell + 1}}(u) = \expsum_{c_{\ell + 2}}(u) = \cdots\), and so
    these values must all be \(0\), as \(u\) cannot have have infinitely many \(c_i\)s with non-zero
    exponent-sum. Similarly, \(\expsum_{c_{-1}}(u) = \expsum_{c_{-2}}(u)
    = \cdots = 0\). Taken with (2) and (5), these imply that \(\expsum_{c_0}(u) = -k\) and
    \(\expsum_{c_\ell}(u) = -1\). Then (3) and (4) can be applied to show
    that \(\expsum_{c_1}(u) = \expsum_{c_{\ell - 1}}(u) = -(k + 1)\). Finally,
    (1) implies that \(\expsum_{c_i}(u) = -(k + 1)\) for all \(i \in \{1, \ldots, \ell - 1\}\).
    We can conclude that \(\mathbf{ab}(u q^\ell) = \ab(c_0^{-k} c_1^{-(k + 1)} \cdots c_{\ell - 1}^{-(k + 1)}
    c_\ell^{-1} q^\ell)\).

    Using the fact that \(c_i = q^i c q^{-i}\) gives
    \[\mathbf{ab}(uq^\ell) 
%    = c^{-k} q c^{-(k + 1)} q^{-1} \cdots
%    q^{\ell - 1} c^{-(k + 1)} q^{-(\ell - 1)} q^\ell c^{-1} q^{-\ell} q^\ell
    = \ab\left(c^{-k} \cdot \prod_{i = 1}^{\ell - 1} (q^i c^{-(k + 1)} q^{-i}) \cdot q^\ell
    c^{-1} q^{-\ell} q^\ell \right).\]
    Moving to the original (free-by-cyclic) presentation for \(G\) via the map \(q \mapsto a^{-1}t\)
    and \(c \mapsto t\) gives
    \begin{align*}
        \mathbf{ab}(uq^\ell) 
        & = \ab\left(t^{-k} \cdot \prod_{i = 1}^{\ell - 1} (a^{-1} t)^i t^{-(k + 1)}
        (a^{-1} t)^{-i} \cdot (a^{-1} t)^\ell t^{-1} \right) \\
        & = \ab \left(a^{-\ell} t^{-k + \ell - 1} \cdot
        \prod_{i = 1}^{\ell - 1}(t^{- (k + 1)}) \right) \\
        & = \ab(a^{-\ell} t^{\ell - (k + 1) - (\ell - 1)(k + 1)}) \\
        & = \ab(a^{-\ell} t^{-k \ell}).
%        & = t^{-k} \cdot \prod_{i = 1}^{\ell - 1} a^{-1} t t^{-(k + 1)}
%        \cdot (a^{-1} t)^{-(\ell - 1)} (a^{-1} t)^\ell t^{-1}\\
%        & = t^{-k} \cdot \prod_{i = 1}^{\ell - 1} a^{-1} t^{-k}
%        \cdot a^{-1} \\
%        & = t^{-k} (a^{-1} t^{-k})^{\ell - 1} a^{-1} t^{-k} \\
%        & = a^{-\ell} t^{-k \ell} 
    \end{align*}
    We have thus shown that \(\ab(C_G(at^k)) \subseteq \{a^\ell t^{k \ell} \mid \ell \in \Z\}\).
    The reverse containment follows since \(\langle at^k \rangle \subseteq C_G(at^k)\) and
    \(\ab(\langle at^k \rangle) = \{a^\ell t^{k \ell} \mid \ell \in \Z\}\).
\end{proof}

The next lemma gives the \(\DPAB{G}\)-definable sets we will need to encode multiplication
of integers into \(\DPAB{G}\). 

\begin{lem}
    \label{Bob-gp-ab-defable}
    Let \(G = \langle a, b, t \mid tat^{-1}
    = ab, tbt^{-1} = b \rangle\). Then the following sets are
    \(\DPAB{G}\)-definable.
    \begin{multicols}{2}
    \begin{enumerate}
        \item \(\langle a, b \rangle\);
        \label{item:ab-ab-defable}
        \item \(\{u t^k \in G \mid \expsum_a(u) = 0\}\);
        \label{item:uexp0-ab-defable}
        \item \(\{ut^k \in G \mid \expsum_a(u) = k\}\);
        \label{item:uexpk-ab-defable}
        \item \(\{(at^k, a^\ell t^{k\ell}) \mid k, \ell \in \Z\}\).
        \label{item:atk-al-tkl-ab-defable}
    \end{enumerate}
    \end{multicols}
\end{lem}

\begin{proof}
The proofs of the four parts are very similar, relying on information about the
abelianisation of several sets, and the fact that \(\mathbf{ab}(b)=1\), which
follows immediately from the presentation of $G$. Additionally, note that \(a\)-exponent sums are well-defined, as \(\ab(G) = \ab(a) \oplus \ab(t)\).

    For \eqref{item:ab-ab-defable}, we claim that \(\langle a, b \rangle = [G,
    G] \langle a \rangle\)w. As a product of \(\DPAB{G}\)-definable sets is
    \(\DPAB{G}\)-definable, and, the latter is \(\DPAB{G}\)-definable
    (\cref{Bob-gp-defable}\eqref{item:a-definable}), showing this equality thus
    shows \eqref{item:ab-ab-defable}. Let \(u \in \langle a, b \rangle\). Then
    \(u a^{-\expsum_a(u)} \in [G, G]\), and so \(u \in [G, G] \langle a
    \rangle\). Conversely, let \(u t^k a^\ell \in [G, G] \langle a \rangle \).
    Then \(k = 0\) and \(\expsum_a(u) = 0\), and so \(ut^k a^\ell = u a^\ell \in
    \langle a, b \rangle\), as required.

    For \eqref{item:uexp0-ab-defable}, we claim that \(\{ut^k \in G \mid \expsum_a(u) = 0\} = [G, G] \langle t
    \rangle\), which is \(\DPAB{G}\)-definable as a product of
    \(\DPAB{G}\)-definable sets, using \cref{Bob-gp-defable}\eqref{item:r-definable}. The proof of this
    claim is identical to the analogous proof in \eqref{item:ab-ab-defable}. 

    For \eqref{item:uexpk-ab-defable}, we will show that \(\{ut^k \in G \mid \expsum_a(u) = k\} = [G,
    G]C_G(at)\). This suffices, as the latter is a product of two
    \(\DPAB{G}\)-definable sets. First note that \cref{ab-cents-Bob-gp} tells us
    that \(\mathbf{ab}(C_G(at)) = \{a^\ell t^\ell \mid \ell \in \Z\}\). 
%    For the forward containment, let \(a^i t^j \in \mathbf{ab}(C_G(at))\). Then there exists
%    \(ut^k \in C_G(at)\) such that \(i = \expsum_a(u)\) and \(j = k\). As
%    \(ut^k \in C_G(a)\), we have
%    \begin{align*}
%        uab^k t^{k + 1}
%        & = ut^k \cdot at\\
%        & = at \cdot ut^k \\
%        & = a \cdot (u)\phi \cdot t^{k + 1},
%    \end{align*}
%    where \(\phi \in \Aut(FG(a, b))\) is defined by \((a) \phi = ab\) and
%    \((b)\phi = b\). Thus \(uab^k = a \cdot (u)\phi\). Thus 
%    \(\expsum_b(a\cdot (u)\phi) = \expsum_b(u a b^k)\). Since for any \(v \in FG(a, b)\)
%    we have \(\expsum_b((v)\phi) = \expsum_b(v) + \expsum_a(v)\), it follows
%    that
%    \begin{align*}
%        \expsum_a(u) + \expsum_b(u) & = \expsum_b((u)\phi) \\ 
%        & = \expsum_b(a \cdot (u) \phi) \\
%        & = \expsum_b(uab^k) \\
%        & = \expsum_b(u) + k.
%    \end{align*}
%    It follows that \(\expsum_a(u) = k\), and so \(i = \expsum_u(a) = k = j\),
%    as required.
%
%    For the reverse containment, note that \(\langle at \rangle \subseteq
%    C_G(at)\), and \(\mathbf{ab}(\langle at \rangle) = \{a^k t^k \mid k \in
%    \Z\}\), as required. We have thus shown that \(\mathbf{ab}(C_G(at)) = \{a^k
%    t^k \mid k \in \Z\}\).
    We now show that \(\{ut^k \in G \mid \expsum_a(u) = k\} = [G, G]C_G(at)\).
    This is equivalent to showing that \(\mathbf{ab}(\{ut^k \in G \mid
    \expsum_a(u) = k\}) = \mathbf{ab}(C_G(at))\). But \(\mathbf{ab}(\{ut^k \in G 
    \mid \expsum_a(u) = k\}) = \{a^k t^k \mid k \in \Z\}\), which we have
    just noted equals \(\mathbf{ab}(C_G(at))\), as required.

    It remains to show \eqref{item:atk-al-tkl-ab-defable}. Let \(\mathcal E_1\) and \(\mathcal E_2\) be systems of equations in \(G\) each with
    a variable \(Z_1\) and \(Z_2\), respectively, such that the set of solutions to \(Z_i\) in
    \(\mathcal E_i\) is \(\langle t \rangle\), 
    which exists by \cref{Bob-gp-defable}\eqref{item:r-definable}. Let \(\mathcal F\) be a system of equations in
    \(G\) with a variable \(V\) such that the set of solutions to \(V\) is \(\langle a \rangle\). Next, note that by \cref{ab-cents-Bob-gp}, we have
    that the set of solutions to \((X, Y)\) in the system 
    \[
        \mathcal E_1 \wedge \mathcal E_2 \wedge (X = aZ_1) \wedge (UX = XU) \wedge (\ab(U) = \ab(Y)) \wedge (Y = V Z_2)
    \]
    is \(\{(a t^k, a^\ell t^{k \ell}) \mid k, \ell \in \Z\}\), as required. 
\end{proof}

We can now prove the main theorem of this section: that \(\DPAB{G}\) is undecidable.
The proof repeatedly applies \cref{Bob-gp-ab-defable} to encode Hilbert's 10th 
problem in the ring of integers into \(\DPAB{G}\).

\begin{theorem}
    Let \(G = \langle a, b, t \mid tat^{-1} = ab, tbt^{-1} = b \rangle\). Then
    \(\DPAB{G}\) is undecidable
\end{theorem}

\begin{proof}
    We show that Hilbert's 10th problem over the ring of integers reduces to
    \(\DPAB{G}\). The copy of \(\Z\) in \(G\) we use is \(\langle a \rangle\),
    which is \(\DPAB{G}\)-definable by \cref{Bob-gp-defable}\eqref{item:a-definable}. To encode
    Hilbert's 10th problem in \(\Z\), it suffices to encode the two types
    of equation \(U + V = W\) and \(UV = W\), where \(U\), \(V\) and \(W\)
    can be either variables or constants. Addition is immediate: we simply
    define new variables/constants \(X_U\), \(X_V\) and \(X_W\), by
    \(X_U = a^U\) if \(U\) is a constant and then add the constraints \(X_U \in \langle a \rangle\)
    (which is \(\DPAB{G}\)-definable) if \(U\) is a variable, along with the
    analogous statements for \(V\) and \(W\), and then add the equation
    \(X_U X_V = X_W\). 

    Multiplication requires the application of \cref{Bob-gp-ab-defable}. We 
    achieve this by intersecting various sets we know to be \(\DPAB{G}\)-definable.
    We have from \cref{Bob-gp-defable}\eqref{item:ab-ab-defable} and \eqref{item:atk-al-tkl-ab-defable}, that \(\{(a^i, a^j, a^m, at^k, a^\ell t^{k\ell}) \mid i, j, k, \ell, m \in \Z\}\) is \(\DPAB{G}\)-definable. So let
    \(X_1, X_2, X_3, X_4\) and \(X_5\) be variables in a system \(\mathcal E\) of equations
    with abelianisation constraints in \(G\) such that the set of solutions to \((X_1, X_2, X_3, X_4,
    X_5)\) in \(\mathcal E\) is the stated set. We then add the additional \(\DPAB{G}\)-definable) constraints
    to \(\mathcal E\):
    \begin{enumerate}
        \item[(i)] \(X_1 (a^{-1} X_4)^{-1} \in \{ut^r \in G \mid \expsum_a(u) = r\}\);
        \item[(ii)] \(X_2 X_5^{-1} \in \{ut^r \in G \mid \expsum_a(u) = 0\}\);
        \item[(iii)] \(X_3 (X_2^{-1} X_5)^{-1} \in \{ut^r \in G \mid \expsum_a(u) = r\}\).
    \end{enumerate}
    The constraints are all \(\DPAB{G}\)-definable, by \cref{Bob-gp-ab-defable}\eqref{item:uexp0-ab-defable} and \eqref{item:uexpk-ab-defable}. Furthermore,
    (i) is equivalent to \(i =k\), (ii) is equivalent to \(j = \ell\), and given (ii), (iii) is
    equivalent to \(m = k \ell\). Thus with these additional constraints the set of solutions
    to \((X_1, X_2, X_3)\) is \(\{a^k, a^\ell, a^{k \ell} \mid k, \ell \in \Z\}\). We have thus
    defined multiplication using the Diophantine Problem with abelianisations constraints in
    \(G\). The undecidability of \(\DPAB{G}\) now follows from the undecidability of
    Hilbert's 10th problem over the ring of integers \cite{Matijasevic}.    
\end{proof}

\section{Cyclically pinched groups}
\label{sec:cyc-pinch}
In this section we fully classify which cyclically pinched groups have decidable
Diophantine Problems with abelianisation constraints. A \emph{cyclically pinched
group $G$} is a free product of two finite rank free groups amalgamated over an
infinite cyclic group; that is, 
\[G=F(a_1, \dots, a_n)\ast_{\mathbb{Z}} F(b_1,
\dots, b_m),
\]
and the amalgamation is of the form $\langle u \rangle=\langle v 
\rangle \cong \mathbb{Z}$, where $u \in F(a_1, \dots, a_n)$ and $v\in F(b_1,
\dots, b_m)$ are non-trivial.  As we will show, the only case when cyclically
pinched groups have a decidable Diophantine Problem with abelianisation
constraints is when they have only two generators, and are thus one-relator
groups with centre. This allows us to apply results from
Section~\ref{sec:fbyc-with-centre}. 

A fairly straightforward combination of existing results, gives that the
Diophantine Problem (without constraints) is decidable in all cyclically pinched
groups.

\begin{cor_no_number}[\ref{DP-cyc-pinch}]
The existential theory (and hence Diophantine Problem) is decidable in all cyclically pinched groups.
\end{cor_no_number}

\begin{proof}
Let \(G = F(a_1, \ldots, a_m) \ast_\Z F(b_1, \ldots, b_m)\) be a cyclically
pinched group across words \(u\) and \(v\). If at least one of $u$  and $v$ is
not a power, then $G$ is hyperbolic by \cite{BestvinaFeighn92}*{Corollary (2), page 99}
%{color{red} AL: do we mean the corollary at the end of the intro of \cite{BestvinaFeighn92}? If so, why can't these groups contain BS subgrousps?}, 
so we can apply
\cite{DahmaniGuirardel}*{Theorem 1} to conclude that \(\DP{G}\) is decidable. If
$u$ and $v$ are both powers, then write $u=\bar{u}^p$ and $v=\bar{v}^q$ with
$\bar{u} \in F(a_1, \dots, a_n)$ and $\bar{v}\in F(b_1, \dots, b_m)$. Then
\[
G=F(a_1, \dots, a_n)\ast_{\langle \bar u \rangle} \langle \bar u, \bar v \mid \bar u^p = \bar v^q \rangle
\ast_{\langle \bar v \rangle} F(b_1, \dots, b_m).
\]
By Dahmani's combination theorem \cite{DahmaniComb}*{Theorem 0.1}, $G$ is hyperbolic relative
to free groups and the middle torus knot group. We have that the existential
theory in free groups is decidable (see for example
\cite{DahmaniGuirardel}*{Theorem 1}), and the torus knot group is a one-relator
group with centre, and thus by \cref{fbyc-embed} is a finite-index subgroup of a
group that is virtually a direct product of hyperbolic groups, and so also has a
decidable existential theory, by \cite{equations_VDP}*{Theorem 3.1}. Hence all
parabolic groups are torsion-free and have decidable existential theories, and
so we can apply \cite{DahmaniRelHyp}*{Theorem 0.1} to conclude that the
existential theory (and hence Diophantine Problem) of \(G\) is decidable.
\end{proof}

    As with many proofs of undecidability of Diophantine-like problems, we begin
    by computing various centralisers.

    \begin{lem}
        \label{cyclically-pinched-cents}
        Let \(G = \langle a_1, \ldots, a_r, b_1, \ldots, b_r \mid u = v
        \rangle\), where \(u \in \{a_1, a_1^{-1}, \ldots, a_n, a_n^{-1}\}^\ast\)
        and \(v \in \{b_1, b_1^{-1}, \ldots, b_s, b_s^{-1}\}^\ast\) be such that
        \(r\geq 2\), \(s \geq 1\) and \(u\) is cyclically reduced and not a power of
        \(a_1\) or \(a_2\). Then
        \begin{enumerate}
            \item \(C_G(a_1) = \langle a_1 \rangle\);
            \item \(C_G(a_2) = \langle a_2 \rangle\);
            \item At least one of the following holds:
            \begin{enumerate}
                \item There exists \(q \in \Z\) such that for all \(k \in \Z
                \setminus \{q\}\), we have \(C_G(a_1a_2^k) = \langle a_1 a_2^k
                \rangle\). In addition, \(C_G(a_1^2 a_2^2) = \langle a_1^2 a_2^2
                \rangle\);
                \item \(C_G(a_1 a_2^k) = \langle a_1 a_2^k \rangle\) for
                all \(k \in \Z\).
            \end{enumerate}
        \end{enumerate}
    \end{lem}

    \begin{proof}
        For (1) ((2) is symmetric), we must first show that \(a_1\) is not
        conjugate (in \(\langle a_1, \ldots, a_r \rangle\)) to any element in
        \(\langle u \rangle\). But if it were, then there would exist
        \(g \in \langle a_1, \ldots, a_r \rangle\) and \(k \in \Z \setminus
        \{0\}\) satisfying \(gu^k g^{-1} = a_1\). It follows that \(a_1\) is a \(k\)th
        power, and so \(k = \pm 1\). Thus \(g u g^{-1} = a_1^{\pm 1}\), and so
        \(u\) is conjugate to a power of \(a_1\), contradicting the fact that
        \(u\) is cyclically reduced and not a power of \(a_1\). Thus \(a_1\) is
        not conjugate to any element in \(\langle u \rangle\). We can therefore apply
        \cite{BarkThesis}*{Proposition 2.3} to conclude that \(C_G(a_1) =
        C_{\langle a_1, \ldots, a_r \rangle} (a_1) = \langle a_1 \rangle\).

        For (3), noting that \(u\) is cyclically reduced, we have two cases:
        either \(u\) is conjugate to a power of \(a_1 a_2^q\), or it is not. If it
        is not, then we can apply the same argument we used to show \(C_G(a) =
        \langle a_1 \rangle\) to conclude that \(a_1 a_2^q\) is not conjugate to a
        power of \(u\), and hence by \cite{BarkThesis}*{Proposition 2.3},
        \(C_G(a_1 a_2) = C_{\langle a_1, \ldots, a_r \rangle}(a_1 a_2) = \langle
        a_1 a_2 \rangle\). Otherwise, we have that \(u = g (a_1 a_2^q)^\ell g^{-1}\).
        In this case, noting that non-trivial powers of \(a_1 a_2^q\) and \(a_1
        a_2^k\) for \(k \in \Z \setminus \{q\}\) are never conjugate in \(\langle a_1, \ldots, a_r \rangle\), we
        have that \(u\) is not conjugate to a power of \(a_1 a_2^k\) for all
        \(k \in \Z \setminus \{q\}\). Applying
        exactly the same argument again, we can conclude that
        \(C_G(a_1 a_2^k) = C_{\langle a_1, \ldots, a_r \rangle}(a_1 a_2^k) = \langle
        a_1 a_2^k \rangle\).
    \end{proof}

    We next show that the Diophantine Problem with abelianisation constraints
    is undecidable for a large class of cyclically pinched groups whose presentations
    satisfy various conditions. As we will see later, it turns out that almost all
    cyclically pinched groups are isomorphic to one of these. The proof uses the
    sets we showed were definable as centralisers, together with the commutator
    subgroup to encode Hilbert's 10th problem over the ring of integers into the
    Diophantine Problem with abelianisation constraints of our cyclically
    pinched group.
    \begin{proposition}
        \label{cyc-pinched-not-power-und}
        Let \(G = \langle a_1, \ldots, a_r, b_1, \ldots, b_s \mid u = v
        \rangle\), where \(u \in \{a_1, a_1^{-1}, \ldots, a_n, a_n^{-1}\}^\ast\)
        and \(v \in \{b_1, b_1^{-1}, \ldots, b_s, b_s^{-1}\}^\ast\) are such that
        \(r\geq 2\), \(s \geq 1\) and \(u\) is cyclically reduced and not a
        power of \(a_1\) or \(a_2\). Then \(\DPAB{G}\) is undecidable. 
    \end{proposition}

    \begin{proof}
        We encode the integers together with addition and the squaring map into
        \(G\) via the cyclic subgroup \(\langle a_2 \rangle\). This can encode
        multiplication, and thus Hilbert's 10th problem over \(\Z\), since 
        for all \(x, y \in \Z\) we have \((x + y)^2 = x^2 + 2xy +
        y^2\), and so the equations \(XY = Z\) and \(2Z + X^2 + Y^2 =
        (X+Y)^2\) are equivalent. Since Hilbert's 10th problem over the ring \(\Z\) is
        undecidable \cite{Matijasevic}, encoding this will give the desired
        undecidability.

        First note that \(\langle a_2 \rangle\) is \(\DPAB{G}\)-definable, as
        a centraliser, by \cref{cyclically-pinched-cents}(2). Encoding addition
        of integers realised as powers of \(a_2\) can be done using
        products in this cyclic subgroup; that is, the equation
        \(Z_1 + Z_2 = Z_3\) in \(\Z\) can be encoded by \(X_1 X_2 = X_3\),
        with the constraints \(X_1, X_2, X_3 \in \langle a_2 \rangle\).
        Since these sets are \(\DPAB{G}\)-definable, this is all \(\DPAB{G}\)-definable.
        To deal with the case that some of \(Z_1\), \(Z_2\) and \(Z_3\)
        are constants, we just replace \(X_i\) with the constant \(a_2^{Z_i}\).

        We now encode the squaring unary operation (on the ring \(\Z\)).
%        From
%        \cref{cyclically-pinched-cents}(3), we have that ther exists
%        \(p \in \{1, 2\}\) such that \(C_G(a_1^p a_2^p) = \langle a_2^p\)
        If \(q\) from \cref{cyclically-pinched-cents}(3)(a) exists, define \(p =
        (q + 1) \mod 2\). If \(q\) doesn't exist (that is, we are in case (b)),
        simply set \(p = 0\). We define a system of equations with
        abelianisation constraints in \(G\) to show that the set \(\{(a_2^{2i +
        p}, a_2^j, a_2^{(2i + p)j} ) \mid i, j \in \Z\}\) is \(\DPAB{G}\)-definable.
        Consider the system
        \begin{align}
            & (T a_2 = a_2 T) \wedge \label{T-eqn}\\ 
            & (X = T^2 a_2^p) \wedge \label{X-eqn}\\
            & (Y (a_1 X) = (a_1 X) Y ) \wedge \label{Y-eqn} \\ 
            & (U a_1 = a_1 U) \wedge \label{U-eqn}\\ 
            & (V a_2 = a_2 V) \wedge \label{V-eqn}\\ 
            & (\ab(Y) = \ab(UV)) \wedge \label{ab-Y-eqn}\\ 
            & \begin{cases}
                Wa_1 a_2 = a_1 a_2 W & q \neq 1 \textrm{ or } q \textrm{ does not exist}\\
                W a_1^2 a_2^2 = a_1^2 a_2^2 W & q = 1.
            \end{cases} \quad \wedge \label{W-eqn} \\
            & (S a_2 = a_2 S) \wedge \label{S-eqn} \\
            & \begin{cases}
                Z = S & q \neq 1 \textrm{ or } q \textrm{ does not exist}\\
                Z = S^2 & q = 1.
            \end{cases} \wedge \label{Z-eqn} \\ 
            & (\ab(W) = \ab(UZ)). \label{ab-W-eqn}
        \end{align}
        We claim that the set of solutions to \((X, Z, V)\) is \(\{(a_2^{2i +
        p},  a_2^j, a_2^{(2i + p)j}) \mid i, j \in \Z\}\). Using the centraliser
        descriptions of \(G\) from \cref{cyclically-pinched-cents}(1) and (2), we have that
        the solutions to \(X\) from \eqref{T-eqn} and \eqref{X-eqn} are
        \(\{a_2^{2i + p} \mid i \in \Z\}\), and so the solutions to \((X, Y)\)
        from \eqref{T-eqn}, \eqref{X-eqn} and \eqref{Y-eqn} are \(\{(a_2^{2i +
        p}, (a_1 a_2^{(2i + p)})^j) \mid i, j \in \Z\}\). We can again use the
        centraliser descriptions to see that the solutions to \eqref{U-eqn} and
        \eqref{V-eqn} are \(\langle a_1 \rangle \) and \(\langle a_2 \rangle\),
        respectively. Combining these with \eqref{ab-Y-eqn} along with the fact
        that the solutions to \((X, Y)\) from \eqref{T-eqn}, \eqref{X-eqn} and
        \eqref{Y-eqn} are \(\{(a_2^{2i +
        p}, (a_1a_2^{(2i + p)})^j) \mid i, j \in \Z\}\), gives that the set of solutions
        to \((X, Y, U, V)\) from these equations is
        \[
            \{(a_2^{2i + p}, (a_1 a_2^{(2i + p)})^j, a_1^j, a_2^{(2i + p)j}) \mid i, j \in \Z\}.
        \]
        At this point we very nearly have that the claimed set is
        \(\DPAB{G}\)-definable, except that the \(j\) appears as the exponent of
        \(a_1\), rather that \(a_2\). The remaining equations deal with this,
        but `copying' the exponent of \(U\) onto \(a_2\) within the exponent of
        \(Z\). To see this, note that regardless of the existence or value of
        \(q\), the equations \eqref{W-eqn}, \eqref{S-eqn}, \eqref{Z-eqn} and
        \eqref{ab-W-eqn} imply that the set of solutions to \(Z\) is contained
        in \(\langle a_2 \rangle\) and \(a_2^j\) is a solution to \(Z\) if and
        only if \(a_1^j\) is a solution to \(U\). We can conclude that the set
        of solutions to \((X, Z, V)\) in the above system is \(\{(a_2^{2i + p},
        a_2^j, a_2^{(2i + p)j}) \mid i, j \in \Z\}\).

        We can now finally encode the squaring function. First, let
        \(\mathcal{E}\) be the above system of equations with abelianisation
        constraints such that the set of solutions to \((X, Z, V)\) is
        \(\{(a_2^{2i + p}, a_2^j, a_2^{(2i + p)j}) \mid i, j \in \Z\}\).
        Then consider the system with all equations in \(\mathcal{E}\),
        together with the equation that \(X = Z^2 a_2^{p}\). Then the set of
        solutions to \((X, Z, V)\) in this new system is \(\{(a_2^{2j + p},
        a_2^{j}, a_2^{4j^2 + 4pj + p^2}) \mid  j \in \Z\}\).
        Further adding the equation \(4Z_2 + 4pZ + p^2 = V\), where
        \(Z_2\) is a new variable, gives that the set of solutions to
        \((Z, Z_2)\) is \(\{(a_2^j, a_2^{j^2}) \mid j \in \Z\}\). We
        have now encoded \(\Z\) with addition and squaring into
        \(\DPAB{G}\), which as discussed at the beginning of
        this proof is enough to encode Hilbert's 10th problem over
        \(\Z\), and thus conclude undecidability.
    \end{proof}

    We can now prove our classification of Diophantine Problems with
    abelianisation constraints in cyclically pinched groups. The proof
    relies heavily on \cref{cyc-pinched-not-power-und} and the decidability
    of the Diophantine Problem with abelianisation constraints in
    one-relator groups with centre and virtually cyclic abelianisation
    \cref{orgwc-decidable-cor}. Most of the proof shows that all
    cyclically pinched groups that are not `obviously' one-relator
    groups with centre, do in fact lie in the class studied in
    \cref{cyc-pinched-not-power-und}.

    \begin{theorem_no_number}[\ref{thm:cycpinDPAB}]
        Let \(G\) be a cyclically pinched group. Then \(\DPAB{G}\) is decidable
        if and only if \(G\) has exactly two generators. In the
        case that \(G\) has two generators, \(\DPAB{G}\) is EDT0L.
    \end{theorem_no_number}

    \begin{proof}
        Case 1: \(G\) has two generators. Then \(G\) is a one-relator group with
        centre, with the presentation \(\langle a_1, b_1 \mid a_1^p = b_1^q
        \rangle\), for some \(p, q \in \Z\). If \(p \neq 1\) or \(q \neq 1\),
        then \(G\) has a virtually cyclic abelianisation, and so
        \cref{orgwc-decidable-cor} tells us that \(\DPAB{G}\) is decidable (and
        EDT0L). Otherwise, \(p = q = 1\), and so \(G \cong \Z^2\). As
        an abelian group, abelianisation constraints in \(\Z^2\) are trivial,
        and so \(\DPAB{\Z^2}\) is decidable (and EDT0L; see for example
        \cite{EvettsLevine22}*{Corollary 3.13}) .

        Case 2: \(G\) has at least three generators. Then \(G\) has the presentation
        \[
            \langle a_1, \ldots, a_r, b_1, \ldots, b_s \mid u = v \rangle,
        \]
        where \(r \geq 2\), \(s \geq 1\), \(u \in \{a_1, a_1^{-1}, \ldots, a_n,
        a_n^{-1}\}^\ast\) and \(v \in \{b_1, b_1^{-1}, \ldots, b_s,
        b_s^{-1}\}^\ast\). We first show that we can assume without loss of
        generality that \(u\) is cyclically reduced. Suppose that \(u \equiv x
        w x^{-1}\), for some freely reduced words \(x, w \in \{a_1, a_1^{-1},
        \ldots, a_r, a_r^{-1}\}^\ast\), with \(w\) cyclically reduced. For each \(i \in \{1, \ldots, r\}\), apply
        the Tiezte transformation to the presentation \(\langle a_1, \ldots, a_r, b_1, \ldots, b_s \mid x w x^{-1} = v \rangle\) that introduces the redundant
        generator \(c_i = x a_i x^{-1}\). Then \(G\) has the presentation:
        \[
            \langle a_1, \ldots, a_r, c_1, \ldots, c_r, b_1, \ldots, b_s \mid
            c_i = x a_i x^{-1}, \bar{w} = v \rangle,
        \]
        where \(\bar{w}\) is the word obtained from \(w\) by replacing each
        occurrence of \(a_i\) with \(c_i\). Then \(\bar{w}\) is cyclically
        reduced. In addition, we can now apply the Tiezte transformation that
        removes the redundant generators \(a_i = x^{-1} c_i x\) for all
        \(i\), giving the following presentation for \(G\):
        \[
            \langle c_1, \ldots, c_r, b_1, \ldots, b_s \mid
            \bar{w} = v \rangle.
        \]
        This is a cyclically pinched presentation where the first of the two
        equal words is cyclically reduced. By returning to our original notation
        of \(a_i\)s and \(u\), we can now assume without loss of generality that
        \(u\) is cyclically reduced.
        
        Since \(u\) is cyclically reduced and the only cyclic permutations of
        powers of \(a_1\) and \(a_2\) are themselves powers of \(a_1\) and
        \(a_2\),  we have that \(u\) is conjugate to a power of \(a_1\) or
        \(a_2\) if and only if \(u\) is equal to a power of \(a_1\) or \(a_2\).
        Thus if \(u\) is not equal to a power of \(a_1\) or \(a_2\), we can
        apply \cref{cyc-pinched-not-power-und} to conclude that \(\DPAB{G}\) is
        undecidable. So now suppose that \(u\) is a power of \(a_1\) or \(a_2\);
        without loss of generality suppose \(u= a_1^p\), where \(p \in \Z
        \setminus \{0\}\). If \(r > 2\) then applying the Tietze transformation
        that swaps the labels \(a_1\) and \(a_3\) gives an isomorphic
        presentation where \(u\) is not a power of \(a_1\) or \(a_2\), and so
        \cref{cyc-pinched-not-power-und} can be applied to give undecidability.
        Similarly, if \(s > 1\), then either \(v\) is a power of some \(b_i\),
        in which case swapping the labels \(a_1\) and \(b_2\) also gives
        undecidability; otherwise \(v\) is not a power of \(b_1\) or \(b_2\), in
        which case swapping the labels \(a_1\) with \(b_1\) and \(a_2\) with
        \(b_2\) gives undecidability. This leaves the case where \(r = 2\) and
        \(s = 1\); that is \(G\) has the presentation
        \[
            \langle a_1, a_2, b_1 \mid a_1^p = b_1^q \rangle,
        \]
        where \(q \in \Z \setminus \{0\}\). Now apply a Tietze transformation
        introducing the new generator \(c = a_1 a_2^{-1}\). Thus \(G \cong
        \langle a_1, a_2, b_1, c \mid a_1^p = b_1^q, c = a_1 a_2^{-1} \rangle\).
        Since \(a_1 = c a_2\), we can apply a Tietze transformation removing the
        redundant generator \(a_1\), giving \(G \cong\langle a_2, b_1, c
        \mid (c a_2)^p = b_1^q \rangle\). Finally,
        applying the Tietze transformation that renames \(c\) to be
        \(a_1\), gives that \(G \cong \langle a_1, a_2, b_1 \mid
        (a_1 a_2)^p = b_1^q \rangle\). Again, \cref{cyc-pinched-not-power-und}
        gives that \(\DPAB{G}\) is undecidable.
    \end{proof}

\section{Remarks and open questions}

In this paper we exhibited several classes of one-relator and free-by-cyclic groups where the Diophantine problem with abelianisation constraints is decidable if the abelianisation is `small' i.e. virtually $\Z$, and groups where \(\DPAB{G}\) is undecidable and the abelianisation has torsion-free rank greater than 2.
There are examples of free by cyclic groups with centre of the form  
$F_n \rtimes_\phi \mathbb{Z}$ and $\mathbb{Z}^2$ abelianisation where we 
do not yet know whether 
\(\DPAB{G}\) 
is decidable. 
For instance we do not know whether the group  
\(\langle a, b, t \mid tat^{-1} = b, tbt^{-1} = a \rangle\) has a decidable  
\(\DPAB{G}\).

%{\color{blue}
%[Bob: I think it would be helpful if around here we give a concrete example of a group of the form $F_n \rtimes_\phi \mathbb{Z}$ with centre and $\mathbb{Z}^2$ abelianisation where we don't yet know whether (DP,ab) is decidable.]
%}
%{\color{purple} Alex: I think \(\langle a, b, t \mid tat^{-1} = b, tbt^{-1} = a \rangle\) would be a good example. This is very similar to \(F_2 \times \Z\), but instead the abelianisation map
%maps \(F_2\) to a copy of \(\Z\). Coincidentally, this is also one-relator.}
%{\color{blue}
%[Bob: Ok, thanks Alex, yes this example would be fine with me.]
%}

It should be noted, however, that the Diophantine Problem itself, so without
constraints, for free-by-cyclic groups remains an open and
difficult problem. Recent related results show that the Diophantine Problem
is undecidable in one-relator groups \cites{BobAlexBG, CFBG}. More
generally, it is also open whether there is an ascending HNN-extensions of free
groups. It has recently been proved \cite{Logan25} that the conjugacy problem is
decidable in ascending HNN-extensions of free groups. 

A class of groups where \(\DP{G}\) is known to be decidable \cite{equations_VDP} but it is unclear whether adding constraints changes this is the class of even dihedral Artin groups.

\begin{qu}
    Is \(\DPAB{G}\) decidable if $G$ is an even dihedral Artin group? Each such group is isomorphic to the Baumslag-Solitar group $BS(p,p)=\langle a, t \mid ta^p t^{-1}=a^p\rangle$, $p\geq 2$.
\end{qu}

The groups $BS(p,p)$ above have the torsion-free rank of the abelianisation equal to $2$. The case of soluble Baumslag-Solitar groups, which have torsion-free rank of the abelianisation equal to $1$, is also open, both in terms of the \(\DP{G}\) and \(\DPAB{G}\). Note that in this case, both of these questions are equivalent
as solvable Baumslag-Solitar groups have finite commutator width.

\begin{qu}
    Is \(\DPAB{G}\) decidable in $BS(1,p)=\langle a, t \mid ta t^{-1}=a^p\rangle$, $p\geq 2$?
\end{qu}

It would be interesting to classify which GBS groups have decidable Diophantine
Problem, but this is likely a difficult problem since it remains open even for
the Baumslag-Solitar groups $BS(1,n)$.  A key theme in this area is
that Decidability of \(\DPAB{G}\) is influenced by the torsion free rank of the
abelianisation of the group under consideration. Indeed, as explained in
\cite{CiobanuGarreta}  the decidability of  \(\DPAB{G}\)  when $G$ has decidable
Diophantine Problem and its abelianisation has torsion-free rank 1 remains an
open problem. More precisely, the following is an open problem:  

\begin{qu}
Does there exist a group $G$ with the following properties: 
\begin{enumerate}
\item $G$ is finitely generated with decidable Diophantine Problem; 
\item  the abelianisation of $G$ is virtually cyclic; and 
\item  \(\DPAB{G}\) is undecidable.  
\end{enumerate}   
\end{qu}
%{\color{red} Don't torus knot groups provide examples of these?}

While undecidability in all our results relies on encoding arithmetic into systems of equations, will the same hold for single equations?

\begin{qu}
    Do our undecidability results hold if we only look at single equations with abelianisation constraints?
\end{qu}

\bibliography{references}

\end{document}